\documentclass[11pt, a4paper, twoside]{amsart}

\makeatletter
\def\@settitle{\begin{center}%
		\baselineskip14\p@\relax
		\normalfont\LARGE\scshape\bfseries
		\@title
	\end{center}%
}

\def\@setauthors{%
  \begingroup
  \def\thanks{\protect\thanks@warning}%
  \trivlist
  \centering\small \@topsep30\p@\relax
  \advance\@topsep by -\baselineskip
  \item\relax
  \author@andify\authors
  \def\\{\protect\linebreak}%
  \authors%
  \ifx\@empty\contribs
  \else
    ,\penalty-3 \space \@setcontribs
    \@closetoccontribs
  \fi
  \endtrivlist
  \endgroup
}

\makeatother

\makeatletter

\def\subsection{\@startsection{subsection}{2}%
	\z@{.5\linespacing\@plus.7\linespacing}{.5\linespacing}%
	{\normalfont\large\bfseries}}

\def\subsubsection{\@startsection{subsubsection}{3}%
	\z@{.5\linespacing\@plus.7\linespacing}{.5\linespacing}%
	{\normalfont\itshape}}

\usepackage[usenames, dvipsnames]{xcolor}
\definecolor{darkblue}{rgb}{0.0, 0.0, 0.45}

\usepackage[colorlinks	= true,
raiselinks	= true,
linkcolor	= darkblue, 
citecolor	= Mahogany,
urlcolor	= ForestGreen,
pdfauthor	= {},
pdftitle	= {},
pdfkeywords	= {},
pdfsubject	= {},
plainpages	= false]{hyperref}

\usepackage{dsfont,amssymb,amsmath, graphicx, multirow} 
\usepackage{amsfonts,dsfont,mathtools, mathrsfs,amsthm} 
\usepackage[amssymb, thickqspace]{SIunits}
\usepackage{algorithm,algorithmicx,fancyhdr}
\usepackage{makecell}

\allowdisplaybreaks
\date{\today}

\DeclareMathOperator*{\argmin}{argmin}

\newcommand{\inner}[2]{\langle #1,#2 \rangle}

\newtheorem{theorem}{Theorem}[section]

\newtheorem{remark}[theorem]{Remark}

\newtheorem{proposition}[theorem]{Proposition}

\graphicspath{{figures/}}

\newcolumntype{?}{!{\vrule width 1.5pt}}
\usepackage{subcaption}
\usepackage{algpseudocode}
\usepackage{aligned-overset}
\usepackage{setspace}
\usepackage{todonotes}

\usepackage{enumerate}
\usepackage[square, numbers]{natbib}
\usepackage{bbm}
\title[Inverse Optimization for Routing Problems]{Inverse Optimization for Routing Problems}

\author{Pedro Zattoni Scroccaro, Piet van Beek, Peyman Mohajerin Esfahani and Bilge Atasoy}%
\thanks{The authors are with the Delft Center for Systems and Control ({\tt\small \{P.ZattoniScroccaro, P.MohajerinEsfahani\}@tudelft.nl, pietvanbeek@live.nl}) and the Department of Maritime and Transport Technology ({\tt\small B.Atasoy@tudelft.nl}), Delft University of Technology, Delft, The Netherlands.}
\thanks{The authors would like to thank Breno A. Beirigo for bringing the Amazon Challenge to their attention. This work is partially supported by the ERC grant TRUST-949796.}

\begin{document}
\maketitle

\begin{abstract}
We propose a method for learning decision-makers' behavior in routing problems using Inverse Optimization (IO). The IO framework falls into the supervised learning category and builds on the premise that the target behavior is an optimizer of an unknown cost function. This cost function is to be learned through historical data, and in the context of routing problems, can be interpreted as the routing preferences of the decision-makers. In this view, the main contributions of this study are to propose an IO methodology with a hypothesis function, loss function, and stochastic first-order algorithm tailored to routing problems. We further test our IO approach in the Amazon Last Mile Routing Research Challenge, where the goal is to learn models that replicate the routing preferences of human drivers, using thousands of real-world routing examples. Our final IO-learned routing model achieves a score that ranks 2nd compared with the 48 models that qualified for the final round of the challenge. Our examples and results showcase the flexibility and real-world potential of the proposed IO methodology to learn from decision-makers' decisions in routing problems.
\end{abstract}


\section{Introduction}\label{sec:introduction}

Last-mile delivery is the last stage of delivery in which shipments are brought to end customers. Optimizing delivery routes is a well-researched topic, but most of the classical approaches for this problem focus on minimizing the total travel time, distance, and/or cost of the routes. However, the routes driven by expert drivers often differ from the routes that minimize a time or distance criterion. This phenomenon is related to the fact that human drivers take many different factors into consideration when choosing routes, e.g., good parking spots, support facilities, gas stations, avoiding narrow streets or streets with slow traffic, etc. This contextual knowledge of expert drivers is hard to model and incorporate into traditional optimization strategies, leading to expert drivers choosing potentially more convenient routes under real-life operational conditions, contradicting the optimized route plans. Thus, developing models that capture and effectively exploit this tactic knowledge could significantly improve the real-world performance of optimization-based routing tools. For instance, in 2021, Amazon.com, Inc. proposed the \textit{Amazon Last Mile Routing Research Challenge} \cite{amazon2021amazon} (referred to
as the Amazon Challenge in the following). For this challenge, Amazon released a dataset of real-world delivery requests and the respective human routes. The goal was for participants to propose novel methods that use this historical data to learn how to route like an expert human driver, thus incorporating their experience and knowledge when routing vehicles for new delivery requests.

In the literature, several approaches have been proposed to incorporate information from historical route data into the planning of new routes. Some of those methods use discrete choice models and the routes of the drivers are used to determine a transition probability matrix \cite{fosgerau2013link}. For instance, in \cite{canoy2019vehicle, canoy2021learn, canoy2021tsp, canoy2024probability}, a Markov chain framework is used to learn the weights associated with each edge of the graph, which are interpreted as the likelihood of that arc appearing in the optimal solution of the routing problem. Other approaches use inverse reinforcement learning to learn a routing policy that approximates the ones from historical data \cite{wulfmeier2017large, liu2020integrating}. The Technical Proceedings of the Amazon Challenge \cite{winkenbach2021technical} contains 31 articles with approaches that were submitted to the Amazon Challenge. Many of these approaches rely on learning specific patterns in the sequence of predefined geographical city zones visited by expert drivers. The paper \cite{wu2022learning} uses a sequential probability model to encode the drivers' behavior and uses a policy iteration method to sample zone sequences from the learned probability model. The paper \cite{pitombeiraneto2021route} develops an Inverse Reinforcement Learning (IRL) approach for the Amazon Challenge, which despite the name, does not share many similarities with our Inverse Optimization (IO) approach. In particular, in this approach, the TSP is interpreted as a Dynamic Programming (DP) problem, thus, the goal of IRL is to learn the stage cost of this DP from example TPS routes. However, DPs are known to suffer from the curse of dimensionality, that is, these problems become intractable to solve when the dimension of the problem becomes too large (such as for the TSP from the challenge). These issues are reflected in the poor performance of the submissions that use IRL. The IRL method in \cite{song2021inverse} is closer to our IO methodology, in the sense that a weight matrix is learned from data. However, different from our IO approach, which learns the entire weight matrix simultaneously, they use a Neural Network to map node features to a single edge weight, thus, not accounting for the features of neighboring edges. A successful approach to tackle the challenge was to adjust the travel time matrix between zones based on patterns observed in the training dataset. In particular, both the second-place \cite{guo2021last} and third-place \cite{arslan2021data} submissions used this approach. Namely, they extracted rules (i.e., patterns observed in the behavior of the human drivers) through descriptive analysis of the training dataset, and based on these rules, they derived ``discouragement multipliers'', which are simply constants that multiply each value of the travel time matrix. These multipliers were tuned so that the TSP routes computed using the modified travel times enforce the rules previously extracted. Our work shares similarities with \cite{guo2021last} and \cite{arslan2021data}, in the sense that it also uses penalization constants to enforce the behaviors observed in the data. However, differently from them, we combine these penalizations with a custom weight matrix \textit{learned} using IO. The IO methodology in this paper can be interpreted as a way to combine information extracted from a descriptive analysis of the data with information automatically learned from the data. The authors in \cite{contreras2021learning} mention IO as a potential method to effectively tackle the challenge, however, due to the complexity of developing a tailored IO methodology for routing problems, the authors instead used standard ML techniques. The approach that won the Amazon Challenge is based on a constrained local search method, where given a new delivery request, they extract precedence and clustering constraints by analyzing similar historical human routes in the training dataset \cite{cook2022constrained}. Thus, their model is \textit{nonparametric}, in the sense that the entire training dataset is required whenever the route for a new delivery request needs to be computed. This is in contrast with our \textit{parametric} IO model, that is, our model is parametrized by a learned vector of parameters, with a dimension that does not depend on the number of examples in the training dataset.

In IO problems, the goal is to model the behavior of an expert agent, which given an exogenous signal, returns a response action. It is assumed that to compute its response, the expert agent solves an optimization problem that depends on the exogenous signal. In this work, we assume that the constraints imposed on the expert are known, but its cost function is unknown. Therefore, the goal of IO is to, given examples of exogenous signals and corresponding expert responses, model the cost function being optimized by the expert. As an example, in a Capacitated Vehicle Routing Problem (CVRP) scenario, the exogenous signal can be a particular set of customers and their respective demands, and the expert's response can be the CVRP routes chosen by the decision-maker to serve these customers and their demands. The papers \cite{burton1992instance, farago2003inverse, barmann2017emulating} use IO to learn the cost matrix of shortest path problems. The paper \cite{chung2012inverse} investigates IO for the Traveling Salesperson Problem (TSP), where they study the problem of, given an edge-weighted complete graph, a single TSP tour, and a TSP solving algorithm, finding a new set of edge weights so that the given tour can be an optimal solution for the algorithm, and is closest to the original weights. Moreover, cutting-plane methods have been proposed to solve general IO for mixed-integer programs \cite{wang2009cutting, duan2011heuristic, bodur2022inverse}, in particular, the authors in \cite{bodur2022inverse} propose the use of \textit{trust regions} to lower the computational cost of generating cuts. IO has also been used to learn household activity patterns \cite{chow2012inverse}, for network learning \cite{chow2014nonlinear, xu2018network}, and more recently for learning complex model predictive control schemes \cite{akhtar2021learning}. For more examples of applications of IO, we refer the reader to the recent review paper \cite{chan2023inverse} and references therein. Regarding our IO methodology, the paper closest to ours is \cite{zattoniscroccaro2023learning}, where the authors propose a general IO framework, together with a general family of first-order optimization algorithms (a.k.a. mirror-descent algorithms) to solve IO problems. In this work, we tailor and extend this general methodology, for the case of routing problems. Our tailored approach is flexible to what type of routing problem the expert is assumed to solve, handles cases when there is a large number of routing examples, as well as cases when solving the routing problem is computationally expensive. More specifically, the main contributions of this paper are summarized as follows:

\begin{enumerate}[(a)]
    \item \textbf{(IO methodology for routing problems)} We propose an IO methodology, which has the following specifications tailored for routing problems:
    
    \begin{enumerate}[(i)]
        \item \textit{Hypothesis class}: We introduce a hypothesis class of affine cost functions with nonnegative cost vectors representing the weights of edges of a graph, along with an affine term that can capture extra desired properties for the model (Section \ref{sec:hypothesis}). This generalizes the linear hypothesis class of \cite{zattoniscroccaro2023learning}.
        \item \textit{Loss function}: We introduce a loss function for IO applied to routing problems (Section \ref{sec:loss_func}). This loss function extends the Augmented Suboptimality Loss of \cite{zattoniscroccaro2023learning} by using our affine hypothesis class. Moreover, exploiting the fact that the decision variables of the routing problem can be modeled using binary variables, we show that the nonconvex cost function of the inner minimization problem of the loss function can be equivalently reformulated as a convex function (Proposition \ref{prop:ASL_reformulation}). This result is of independent interest since this reformulation can be used for any IO problem with binary decision variables.
        \item \textit{First-order algorithm}: We also design a first-order algorithm specialized to minimize our tailored IO loss function, and is particularly efficient for IO problems with large datasets and with computationally expensive decision problems, e.g., large VRPs (Section \ref{sec:algorithm}).  Compared to the SAMD algorithm proposed in \cite{zattoniscroccaro2023learning} for general IO problems, our algorithm tailors it to our affine hypothesis function and new loss function reformulation and also uses a ``\textit{reshuffled}'' sampling strategy, which improves its empirical performance compared to the uniform sampling employed by the SAMD algorithm  (Section \ref{sec:IO_for_VRPTW}).
        \item \textit{Modeling flexibility}: We showcase the flexibility of our IO methodology by demonstrating how three specific instances of routing problems (CVRP, VRPTW, and TSP) can be modeled using our framework (Sections \ref{sec:IO_for_CVRP},  \ref{sec:IO_for_VRPTW}, and \ref{sec:IO_for_TSP}). We present numerical results that give intuition on how our tailored algorithm works for routing problems (Figure \ref{fig:OI_for_SCVRP}), as well as its efficacy in handling large routing problems (Figure \ref{fig:vrptw_out}).
    \end{enumerate}  
    
    \item \textbf{(Application to the Amazon Challenge)} We evaluate our IO methodology on the Amazon Challenge, namely, we learn the drivers' preferences in terms of geographical city zones using IO. We present results for a general IO approach as well as for the tailored approach developed in this paper, showcasing how insights about the structure of the problem at hand can be seamlessly integrated into our IO methodology, illustrating its flexibility and modeling power (Sections \ref{sec:IO_amazon}). Our approach achieves a final Amazon score of \textbf{0.0302}, which ranks 2nd compared to the 48 models that qualified for the final round of the Amazon Challenge (Figure \ref{fig:score_leaderboard}). Moreover, using an approximate TSP solver and a fraction of the training dataset, we can learn a good routing model in just a few minutes, demonstrating the possibility of using our IO approach for real-time learning problems (Table \ref{table:summary_results}). All of our experiments are reproducible, and the underlying source code is available in \cite{zattoniscroccaro2023amazon}.
\end{enumerate}

The rest of the paper is organized as follows. In the remainder of this section, we define the mathematical notation used in this paper. In Section \ref{sec:inverse}, we introduce the IO methodology used in this paper and our IO approach for routing problems. In Section \ref{sec:modelling}, we present modeling examples for CVRPs, VRPTWs, and TSPs. In Section \ref{sec:challenge} we introduce the Amazon Challenge, its datasets, objective, scoring metric, and our complete IO approach to tackle it. In Section \ref{sec:numerical} we present our numerical results for the Amazon Challenge, as well as further numerical results.

\textbf{Notation.} For vectors $x,y \in \mathbb{R}^m$, $x \odot y$, $\exp(x)$ and $\max(x,y)$ mean element-wise multiplication, element-wise exponentiation and element-wise maximum, respectively. The Euclidean inner product between two vectors $x,y \in \mathbb{R}^m$ is denoted by $\inner{x}{y}$. The set of integers from $1$ to $N$ is denoted as $[N]$. A set of indexed values is compactly denoted by $\{x^{[i]}\}_{i=1}^N \coloneqq \{x^{[1]},\ldots,x^{[N]}\}$. Given a set $A$, we denote its complement by $\bar{A}$ and its cardinality by $|A|$.


\section{Inverse Optimization}
\label{sec:inverse}

In this section, we give a brief introduction to the IO methodology used in this paper and describe our IO approach to learning from routing problems. Let us begin by formalizing the IO problem. Consider an exogenous signal $\hat{s} \in \mathbb{S}$, where $\mathbb{S}$ is the signal space. Given a signal, an expert agent is assumed to solve the following parametric optimization problem to compute its response action:
\begin{equation} \label{eq:expert_problem}
    \min_{x \in \mathbb{X}(\hat{s})} F(\hat{s},x),
\end{equation}
where $\mathbb{X}(\hat{s})$ is the expert's known constraint set, $F: \mathbb{S} \times \mathbb{X} \to \mathbb{R}$ is the expert's unknown cost function, where we define $\mathbb{X} \coloneqq \bigcup_{\hat{s} \in \mathbb{S}} \mathbb{X}(\hat{s})$. In our IO formulation, the signal space $\mathbb{S}$ may contain any information that the expert uses to solve the optimization problem \eqref{eq:expert_problem}. For example, in the context of routing problems, the signal may contain the demands of customers, time windows for the service of customers, the set of customers that need to be served, time of the day, day of the week, weather information, etc. Since it would not be practical to formally (i.e., mathematically) define a signal space that contains all possible types of signals, we leave it as a general signal space $\mathbb{S}$. The expert's decision $\hat{x}$ is chosen from the set of optimizers of \eqref{eq:expert_problem}, i.e., $\hat{x} \in \argmin_{x \in \mathbb{X}(\hat{s})} F(\hat{s},x)$. Assume we have access to $N$ pairs of exogenous signals and respective expert optimal decisions $\{(\hat{s}^{[i]}, \hat{x}^{[i]})\}_{i=1}^N$, that is,
\begin{equation*}
    \hat{x}^{[i]} \in \argmin_{x \in \mathbb{X}(\hat{s}^{[i]})} F(\hat{s}^{[i]},x) \quad \forall i \in [N],
\end{equation*}
where we use the hat notation ``$\hat{\cdot}$'' to indicate signal-response data (e.g., $\hat{s}$ and $\hat{x}$). When using a dataset of signal-response data, we use the superscript ``$^{[i]}$'' to refer to the $i$'th pair of the dataset, e.g., $\hat{s}^{[i]}$ and $\hat{x}^{[i]}$. Using this data, our goal is to learn a cost function that, when optimized for the same exogenous signal, (approximately) reproduces the expert's actions. For a more detailed discussion on the formalization of IO problems, please refer to \cite{mohajerin2018data, chan2023inverse, zattoniscroccaro2023learning}.

\subsection{Affine hypothesis class}
\label{sec:hypothesis}

Since we can only search for cost functions in a restricted function space and given our focus on routing problems, in this work we consider cost functions in an \textit{affine hypothesis space with a nonnegative cost vector}
\begin{equation}
\label{eq:hypothesis_class}
    \mathcal{H}_\theta \coloneqq \left\{\inner{\theta}{x} + h(\hat{s},x): \theta \geq 0 \right\},
\end{equation}
where $\theta \in \mathbb{R}^p$ is the cost vector that parametrizes the cost function, and the affine term $h: \mathbb{S} \times \mathbb{X} \to \mathbb{R}$ is a function that can be used to model terms in the hypothesis function $\inner{\theta}{x} + h(\hat{s},x)$ that do not depend on $\theta$. This affine function class generalizes the standard linear hypotheses common in the literature of IO and is a key component of our IO methodology to achieve state-of-the-art results in real-world problems (Section \ref{sec:IO_amazon}). Moreover, we consider nonnegative cost vectors because, for routing problems, they represent the weights of the edges of a graph. For instance, given a complete graph with $n$ nodes, common cost functions to routing problems are the \textit{two-index} or \textit{three-index} formulations
$$
\inner{\theta}{x} = \sum_{i=1}^n\sum_{j=1}^n \theta_{ij} x_{ij} \quad \text{and} \quad \inner{\theta}{x} = \sum_{i=1}^n\sum_{j=1}^n\sum_{k=1}^K \theta_{ij} x_{ijk},
$$
where $x_{ij}$ and $x_{ijk}$ are binary variables equal to $1$ if the edge connecting node $i$ to node $j$ is used in the route, and $0$ otherwise (for the three-index formulation, we have an extra index $k$ specifying which of the $K$ available vehicles uses the edge) \cite{toth2002vehicle}. Moreover, we could also have an affine term, for instance, $h(\hat{s},x) =  \sum_{i=1}^n\sum_{j=1}^n M_{ij}(\hat{s}) x_{ij}$ in the cost function, where the term $M_{ij}(\hat{s}) \in \mathbb{R}$ can be used to encode some behavior we would like to enforce in the model. In summary, our goal with IO is to learn a cost vector $\theta$ such that when solving the \textit{Forward Optimization Problem} (FOP)
\begin{equation}
\label{eq:FOP}
    \text{FOP}(\theta, \hat{s}) \coloneqq \arg\min_{x \in \mathbb{X}(\hat{s})} \big\{ \inner{\theta}{x} + h(\hat{s},x) \big\},
\end{equation}
we can reproduce (or approximate) the response the expert would have taken when solving the unknown optimization problem \eqref{eq:expert_problem}, given the same signal $\hat{s}$.

\subsection{Tailored loss function}
\label{sec:loss_func}

Given a signal-response dataset $\{(\hat{s}^{[i]}, \hat{x}^{[i]})\}_{i=1}^N$, in this work we propose to solve the IO problem (i.e., find a parameter vector $\theta$) by solving a loss minimization problem:
\begin{equation}
\label{eq:loss_minimization}
    \min_{\theta \geq 0} \ \frac{1}{N}\sum_{i=1}^N \ell_\theta (\hat{s}^{[i]}, \hat{x}^{[i]}),
\end{equation}
where $\ell_\theta : \mathbb{S} \times \mathbb{X} \to \mathbb{R}$ is the \textit{loss function}. Using the affine hypothesis class \eqref{eq:hypothesis_class}, we propose the following loss function
\begin{equation}
    \label{eq:loss_function}
    \ell_\theta(\hat{s},\hat{x}) \coloneqq \inner{\theta}{\hat{x}} + h(\hat{s},\hat{x}) - \min_{x \in \mathbb{X}(\hat{s})} \big\{ \inner{\theta + 2\hat{x}- \mathbbm{1}}{x} + h(\hat{s},x) - \inner{\mathbbm{1}}{\hat{x}} \big\},
\end{equation}
where $\mathbbm{1} \in \mathbb{R}^p$ is the all-ones vector. The loss function \eqref{eq:loss_function} is an extension of the Augmented Suboptimality Loss (ASL) proposed in \cite{zattoniscroccaro2023learning}, differing from the ASL in two ways: (i) it uses the affine hypothesis class introduced in section \ref{sec:hypothesis}, which allows us to effectively use it for a wider range of practical problems (e.g., the Amazon Challenge), and (ii) its inner minimization problem has a convex objective function w.r.t. to $x$ (assuming $h$ is convex in $x$), in contrast to the case for the ASL, which is nonconvex general. Having an inner minimization problem with convex cost makes its use much more practical since the inner optimization problem has to be solved to evaluate or compute gradients of \eqref{eq:loss_function}. For example, when using first-order methods to optimize it, the inner minimization problem must be solved at each iteration of the algorithm (e.g., see Algorithm \ref{alg:first_order}). This ``nonconvex to convex'' reformulation is possible by exploiting the fact that routing problems can be modeled using binary decision variables (e.g., $x_{ij} = 1$ if the edge connecting nodes $i$ and $j$ is used, and $x_{ij} = 0$ otherwise). This reformulation is formalized in Proposition \ref{prop:ASL_reformulation}.
\begin{proposition}[Connection between the ASL and \eqref{eq:loss_function}]
\label{prop:ASL_reformulation}
    Assume $\mathbb{X} \subseteq \{0, 1\}^p$, that is, the decision variables of the FOP are binary. Then, the loss function \eqref{eq:loss_function} is equivalent to the ASL $\ell^{\emph{ASL}}_\theta(\hat{s},\hat{x}) = \inner{\theta}{\phi(\hat{s},\hat{x})} - \min_{x \in \mathbb{X}(\hat{s})} \big\{ \inner{\theta}{\phi(\hat{s},x)} - d(\hat{x},x) \big\}$, if the linear hypothesis $\inner{\theta}{\phi(\hat{s},\hat{x})}$ is substituted by the affine hypothesis $\inner{\theta}{\hat{x}} + h(\hat{s},\hat{x})$, and the distance function $d(\hat{x},x) = \| \hat{x} - x \|_1$.
\end{proposition}
\begin{proof}
    The ASL with the linear hypothesis substituted by the affine hypothesis and $d(\hat{x},x) = \| \hat{x} - x \|_1$ is equal to $\inner{\theta}{\hat{x}} + h(\hat{s},\hat{x}) - \min_{x \in \mathbb{X}(\hat{s})} \big\{ \inner{\theta}{x} + h(\hat{s},x) - \| \hat{x} - x \|_1 \big\}$. Next, notice that for binary variables $a,b \in \{0,1\}$, we have the identity $|a - b| = (1-a)b + (1-b)a$. Thus, for two binary vectors $\hat{x},x \in \{0, 1\}^p$, using the definition of the $\ell_1$-norm, we have that $\|\hat{x} - x\|_1 = \inner{\mathbbm{1} - 2\hat{x}}{x} + \inner{\mathbbm{1}}{\hat{x}}$.
\end{proof}

\subsection{First-order algorithm}
\label{sec:algorithm}

In this work, we propose to solve problem \eqref{eq:loss_minimization} using a stochastic first-order algorithm. In particular, our algorithm tailors and extends the SAMD algorithm from \cite{zattoniscroccaro2023learning}, as it uses update steps tailored to the proposed loss function \eqref{eq:loss_function} with a nonnegative cost vector. Moreover, it exploits the finite sum structure of the problem (i.e., the sum over the $N$ examples) by using a single training example per iteration of the algorithm. To do so, we propose a reshuffled sampling strategy, which empirically outperforms the uniform sampling strategy of the SAMD algorithm (see results in sections \ref{sec:IO_for_VRPTW} and \ref{sec:IO_amazon}).

Before presenting our algorithm, we discuss how to compute subgradients of the loss function \eqref{eq:loss_function}, which is necessary to use first-order methods to minimize it. To this end, we define
\begin{equation} 
\label{eq:AFOP}
    \text{A-FOP}(\theta, \hat{s}, \hat{x}) \coloneqq \arg\min_{x \in \mathbb{X}(\hat{s})} \big\{ \inner{\theta + 2\hat{x} - \mathbbm{1}}{x} + h(\hat{s},x) \big\},
\end{equation}
that is, the set of optimizers of the FOP with \textit{augmented} edge weights $\theta + 2\hat{x} - \mathbbm{1}$ instead of $\theta$. Being able to solve the augmented FOP \eqref{eq:AFOP} is important because to compute a subgradient of the loss \eqref{eq:loss_function} (and thus, a subgradient of \eqref{eq:loss_minimization}), we need to compute an element of $\text{A-FOP}(\theta, \hat{s}, \hat{x})$, which follows from Danskin's theorem \cite[Section B.5]{bertsekas2008nonlinear}. Here we emphasize an important consequence of the reformulation that led to our tailored loss function: assuming the affine term is linear in $x$, say $h(\hat{s},x) = \inner{M(\hat{s})}{x}$, solving the A-FOP has the same complexity as solving the FOP. For example, if the FOP is a TSP with edge weights $\theta$, then the A-FOP is also a TSP, but with augmented edge weights $\theta + 2\hat{x} - \mathbbm{1} + M(\hat{s})$. This is of particular practical interest since the same solver can be used both for \textit{learning} the model (i.e., for the A-FOP, which we must be able to solve to evaluate \eqref{eq:loss_function}, thus, to solve \eqref{eq:loss_minimization}) and for \textit{using} the model (i.e., for the FOP).

\begin{algorithm}
\setstretch{1.2}
\caption{Reshuffled stochastic first-order algorithm}
\label{alg:first_order}
\begin{algorithmic}[1]
\State \textbf{Input:} Step-size sequence $\{ \eta_t \}_{t=1}^T$, initial point $\theta_1^{[1]} \geq 0$, number of epochs $T$, and dataset $\{(\hat{s}^{[i]}, \hat{x}^{[i]})\}_{i=1}^N$
\For{$t=1, \ldots, T$}
\State Sample $\{\pi_1, \ldots, \pi_N\}$, a permutation of $\{1, \ldots, N\}$
\For{$i=1, \ldots, N$}
\State $x^\star \in \text{A-FOP}\left(\theta_t^{[i]}, \hat{s}^{[\pi_i]}, \hat{x}^{[\pi_i]}\right)$
\State $g = \hat{x}^{[\pi_i]} - x^\star$
\State
\begin{align*}
    \theta_t^{[i+1]} =
    \begin{cases}
    \theta_t^{[i]} \odot \exp(-\eta_t g) \quad & \text{(exponentiated update)} \\
    \text{OR} \\
    \max \left\{ 0, \ \theta_t^{[i]} - \eta_t g \right\} \quad & \text{(standard update)}
    \end{cases}  
\end{align*}
\EndFor
\State $\theta_{t+1}^{[1]} = \theta_{t}^{[N+1]}$
\EndFor
\State \textbf{Output:} $\left\{ \theta_t^{[N+1]} \right\}_{t=1}^T$
\end{algorithmic}
\end{algorithm}

Algorithm \ref{alg:first_order} shows our \textit{reshuffled stochastic first-order algorithm} to solve Problem \eqref{eq:loss_minimization} using the loss function \eqref{eq:loss_function}. The algorithm runs for $T$ epochs. The number of epochs $T$ should be viewed as an input parameter of Algorithm \ref{alg:first_order}. In practice, one can choose $T$ to be as large as possible and then monitor the performance of the learned model after each epoch of the algorithm. This way, the algorithm is evaluated for all epochs up until $T$, and we can choose the model with the best performance. Alternatively, we can use Algorithm \ref{alg:first_order} without a predefined number of epochs $T$ and run it until a stopping criterion is reached. In practice, this could be implemented by running the algorithm until the difference in the test dataset performance (or any other performance metric) of the models from epochs $t$ and $t+1$ is smaller than a minimum value. For instance, in Figure \ref{fig:dataset_size}, we can see that between epochs $T=4$ and $T=5$, the Amazon score of the learned models does not change much, thus, we could use it to stop the algorithm. In our numerical experiments, we chose a predefined  $T$ and monitored the performance of the model after each epoch. At the beginning of each epoch, we sample a permutation of $[N]$ (line 3), which simply means that we shuffle the order of the examples in the dataset. This is known as \textit{random reshuffling}, as has been shown to perform better in practice compared to standard uniform stochastic sampling \cite{mishchenko2020random}. Moreover, since our random reshuffling strategy uses only one example per update step, it is particularly efficient for problems with large datasets. Next, for each epoch, we perform one update step for each example in the dataset. In particular, in line 5 we compute one element of A-FOP, and in line 6 we compute a subgradient of the loss function \eqref{eq:loss_function}. For the update step (line 7), we offer two possibilities: (i) \textit{exponentiated updates}, which are inspired by the exponentiated subgradient algorithm of \cite{kivinen1997exponentiated} and are specialized for optimization problems with nonnegative variables, and (ii) \textit{standard updates}, which can be interpreted as using the standard projected subgradient method, projecting onto the nonnegative cone. In practice, the question of what update step is the best should be answered on a case-by-case basis. An important component of our algorithm (and of first-order algorithms in general) is the step size $\eta_t$. Common choices are $\eta_t = c/\sqrt{t}$, $\eta_t = c/t$, or $\eta_t = c$, for some fixed constant $c > 0$. Finally, to turn the output of the algorithm $\big\{ \theta_t^{[N+1]} \big\}_{t=1}^T$ into a single cost vector, we can use standard methods such as $\theta = \theta_T^{[N+1]}$ (last iterate), $\theta = \frac{1}{T} \sum_{t=1}^T \theta_t^{[N+1]}$ (average) or $\theta = \frac{2}{T(T+1)}\sum_{t=1}^T t \theta_t^{[N+1]}$ (weighted average) \cite{lacoste2012simpler}.

\begin{remark}[Approximate A-FOP]
\label{remark:approximate_A-FOP}
    Notice that an element of A-FOP needs to be computed at each iteration of Algorithm \ref{alg:first_order} (line 5). However, if the A-FOP is a hard combinatorial problem (e.g., a large TSP or VRPTW), it may not be computationally feasible to solve it to optimality multiple times. Thus, in practice, one may use an \emph{approximate A-FOP}, that is, in line 5 of Algorithm \ref{alg:first_order} we compute an approximate solution to the augmented FOP instead of an optimal one. Fortunately, an approximate solution of the A-FOP can be used to construct an \emph{approximate subgradient} \citep[Example 3.3.1]{bertsekas2015convex}, which in turn can be used to compute an approximate solution of Problem \eqref{eq:loss_minimization} \cite{zattoniscroccaro2023learning}. In practice, using approximate solvers may lead to a much faster learning algorithm, in exchange for a possibly worse learned model. This trade-off is explored in the numerical results of Section \ref{sec:complexity}.
\end{remark}


\section{Modeling Examples}
\label{sec:modelling}

Next, we present three examples of how our IO methodology can be used for learning from routing data. Namely, we first exemplify how a CVRP scenario can be modeled with our IO methodology, and present a simple numerical example to illustrate the intuition behind how Algorithm \ref{alg:first_order} works. Second, we show how a larger VRPTW scenario can be modeled with our IO methodology, and present numerical results using data generated from real-world instances. Third, we define a class of TSPs, which will later be used to formalize the Amazon Challenge as an IO problem.

\subsection{IO for CVRPs}
\label{sec:IO_for_CVRP}

We define the $K$-vehicle Symmetric Capacitated Vehicle Routing Problem (SCVRP) as
\begin{equation}
\label{eq:SCVRP}
\begin{aligned}
    \min_{x_e \in \{0,1\}} \quad & \sum_{e \in E} w_e x_e, \\
    \text{s.t.} \quad & \sum_{e \in \delta(i)} x_e = 2  & \forall i \in V \setminus \{0\} \\
     & \sum_{e \in \delta(0)} x_e = 2K \\
     & \sum_{e \in \delta(S)} x_e \geq 2 r(S, D, c) \quad & \forall S \subset V \setminus \{0\}, \hspace{1mm} S \neq \emptyset,
\end{aligned} 
\end{equation}
where $\mathcal{G} = (V, E, W)$ is an edge-weighted graph, with node set $V$ (node $0$ being the depot), undirected edges $E$, and edge weights $W$. For this problem, each node $i \in V$ represents a customer with demand $d_i \in D$. There are $K$ vehicles, each with a capacity of $c$. Given a set $S \subset V$, let $\delta(S)$ denote the set of edges that have only one endpoint in S. Moreover, given a set $S \subset V \setminus \{0\}$, we denote by $r(S, D, c)$ the minimum number of vehicles with capacity $c$ needed to serve the demands of all customers in $S$. The $x_e$'s are binary variables equal to $1$ if the edge $e \in E$ is used in the solution, and equal to $0$ otherwise, and $w_e \in W$ is the weight of edge $e \in E$ \cite{toth2002vehicle}.

\begin{figure}
    \centering
    \captionsetup[subfigure]{width=0.96\linewidth}%
    \begin{subfigure}[t]{.4\linewidth}
        \includegraphics[width = \linewidth]{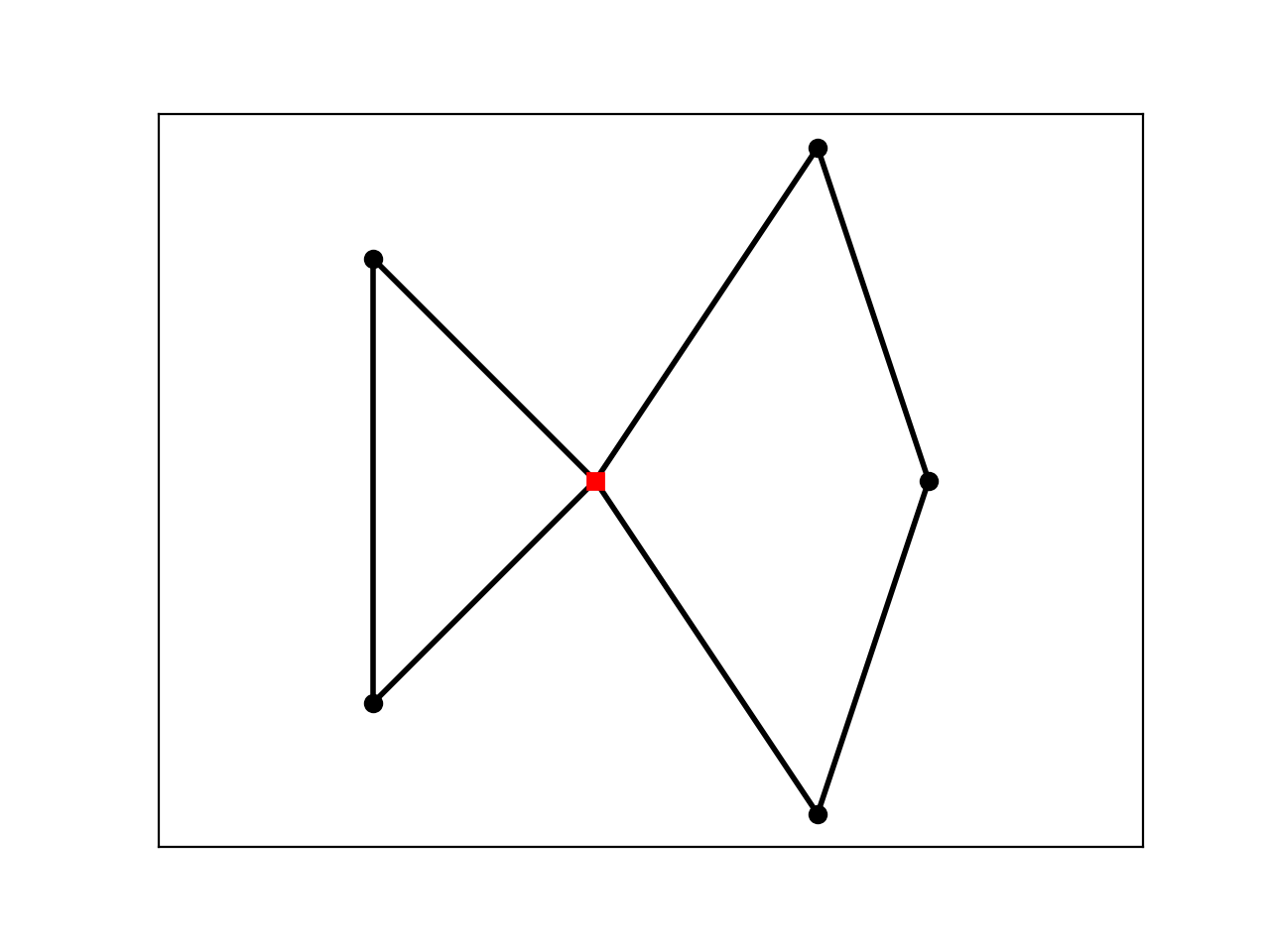}
        \caption{Optimal SCVRP routes using weights $w_e$ based on Euclidean distances.}
        \label{fig:true_weights_routes}
    \end{subfigure}
    \begin{subfigure}[t]{.4\linewidth}
        \includegraphics[width = \linewidth]{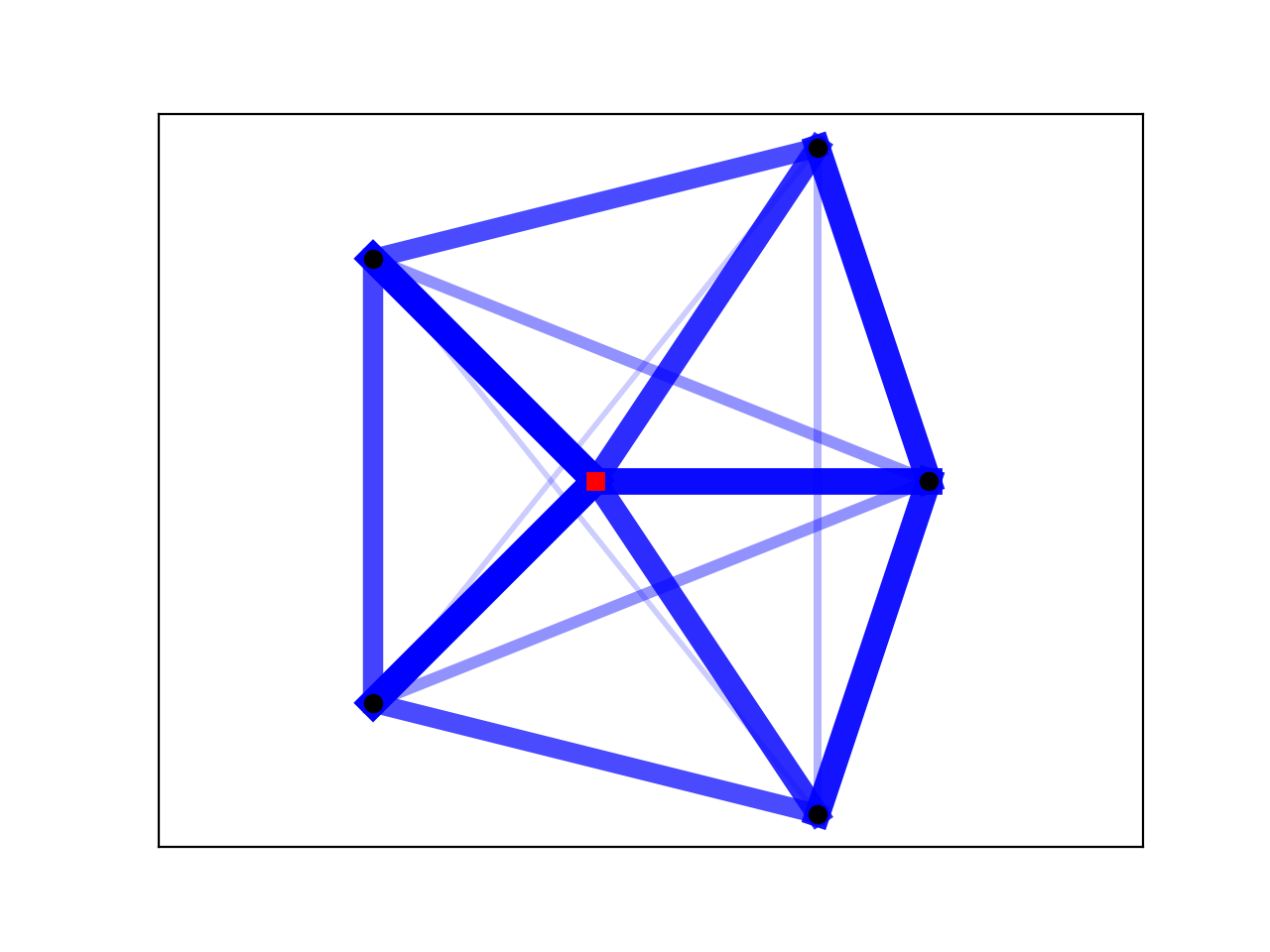}
        \caption{Representation of the weights of each edge of the graph, where the smaller the weight, the thicker and darker the edge.}
        \label{fig:true_weights}
    \end{subfigure}
\caption{Optimal SCVRP tour and representation of graph weights.}
\label{fig:weights_and_routes}
\end{figure}

\begin{figure}
    \centering
    \captionsetup[subfigure]{width=0.96\linewidth}%
    \begin{subfigure}[t]{.32\linewidth}
        \includegraphics[width = \linewidth]{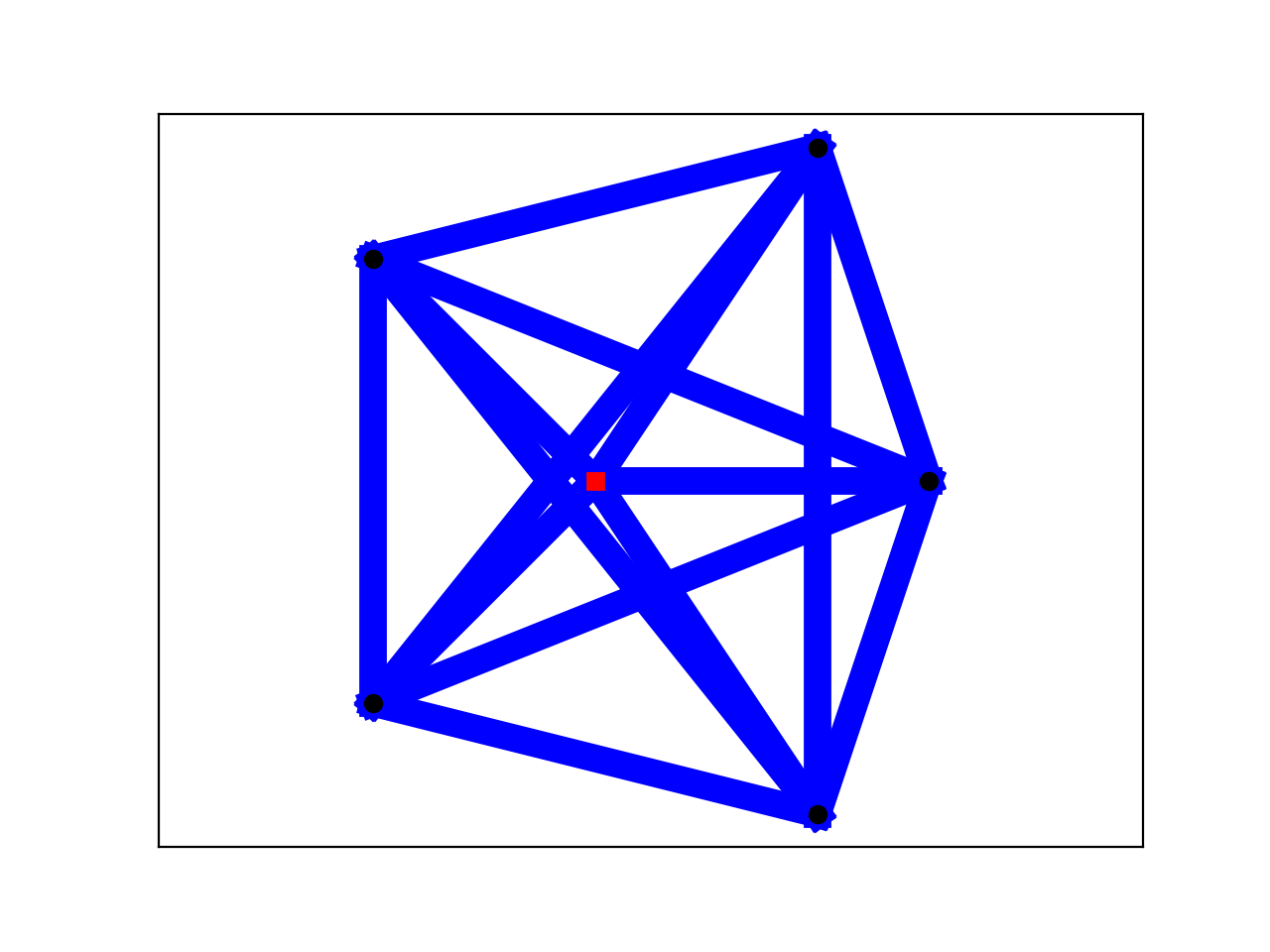}
        \caption{$\theta_1$.}
        \label{fig:theta_1}
    \end{subfigure}
    \begin{subfigure}[t]{.32\linewidth}
        \includegraphics[width = \linewidth]{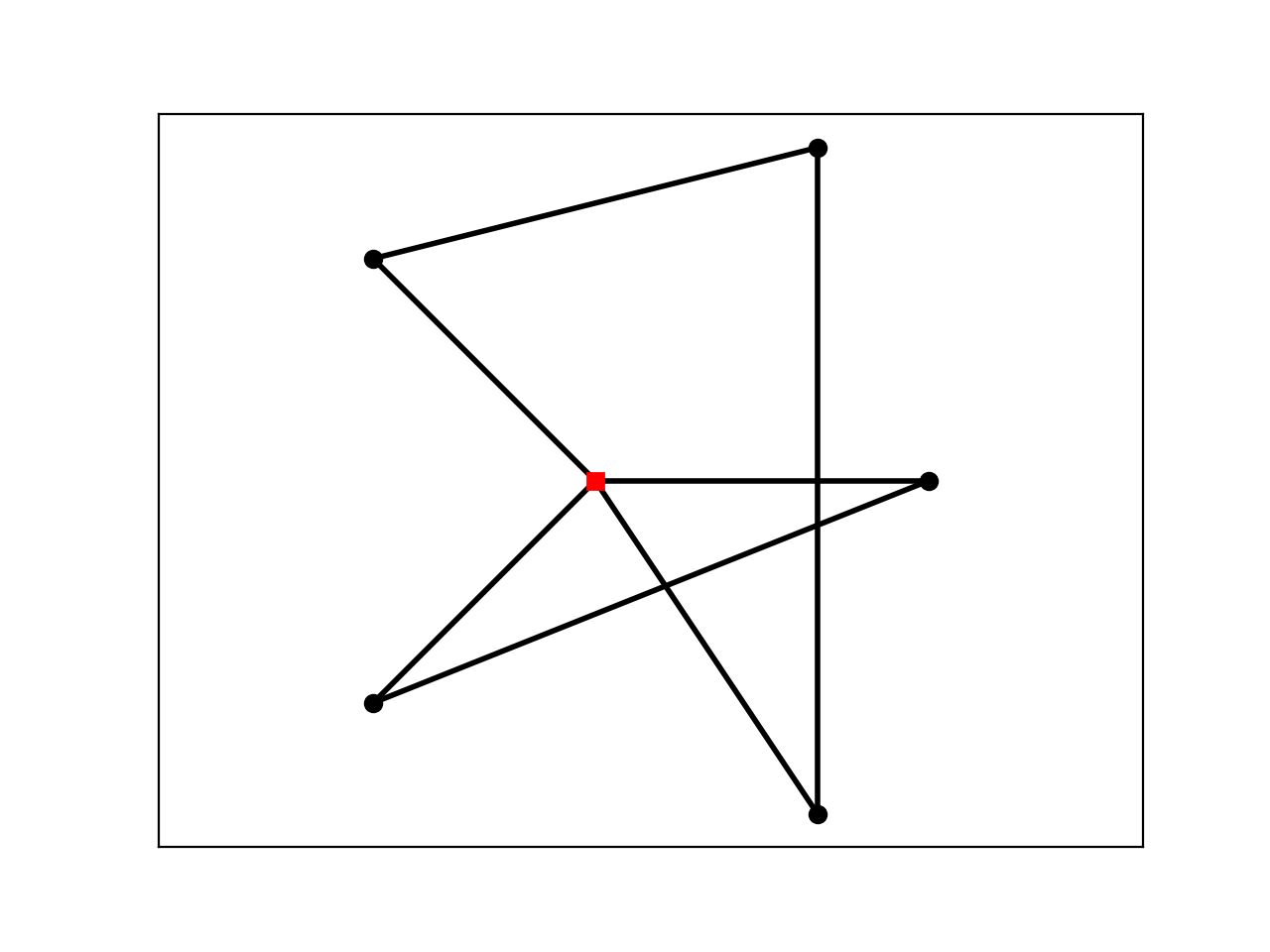}
        \caption{Optimal SCVRP routes using the weights $\theta_1$.}
        \label{fig:theta_1_routes}
    \end{subfigure}
    \begin{subfigure}[t]{.32\linewidth}
        \includegraphics[width = \linewidth]{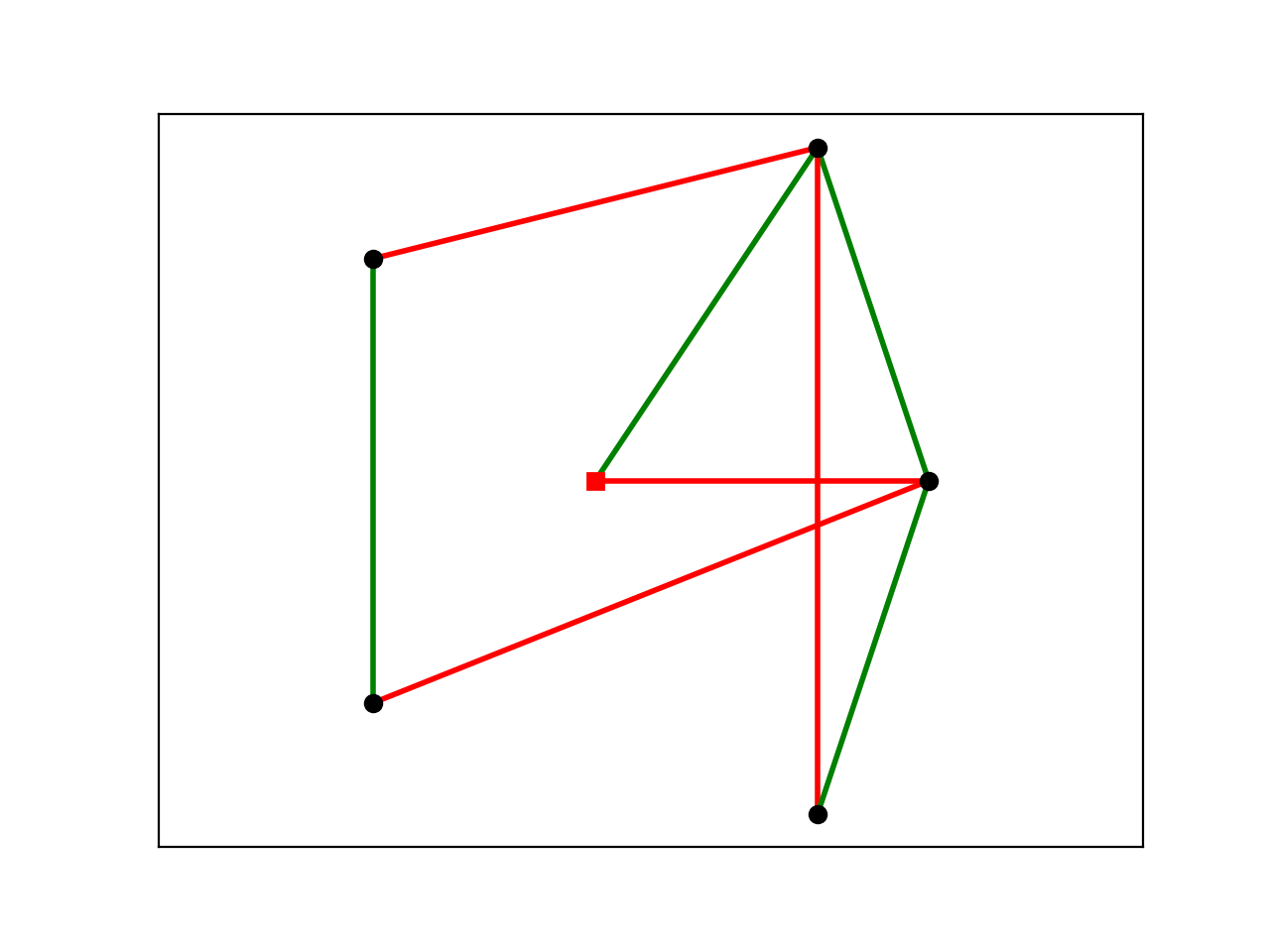}
        \caption{Difference between true optimal route (Figure \ref{fig:weights_and_routes}) and learned route (Figure \ref{fig:theta_1_routes}).}
        \label{fig:grad_1}
    \end{subfigure} \\
    \begin{subfigure}[t]{.32\linewidth}
        \includegraphics[width = \linewidth]{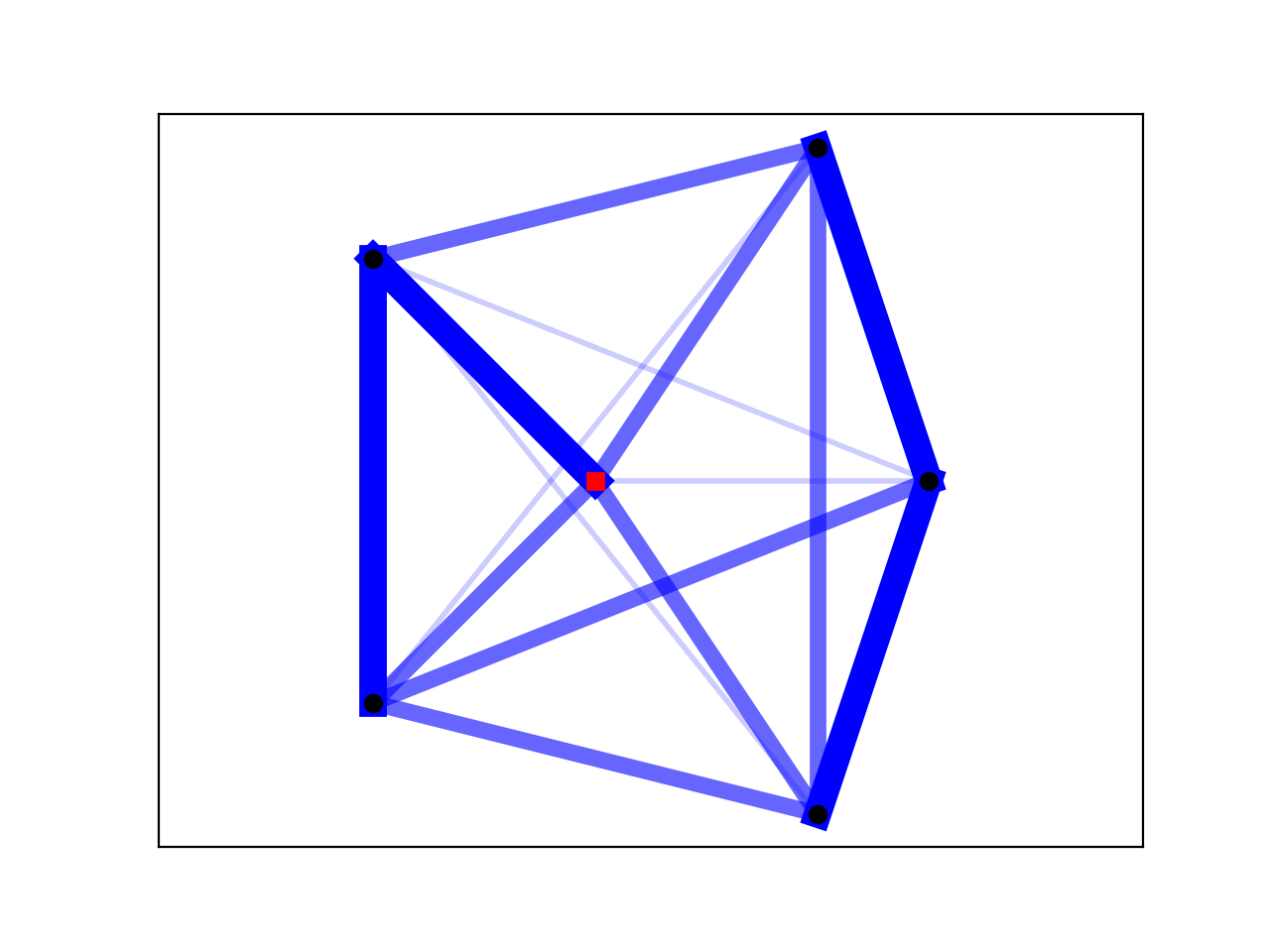}
        \caption{$\theta_2$.}
        \label{fig:theta_2}
    \end{subfigure}
    \begin{subfigure}[t]{.32\linewidth}
        \includegraphics[width = \linewidth]{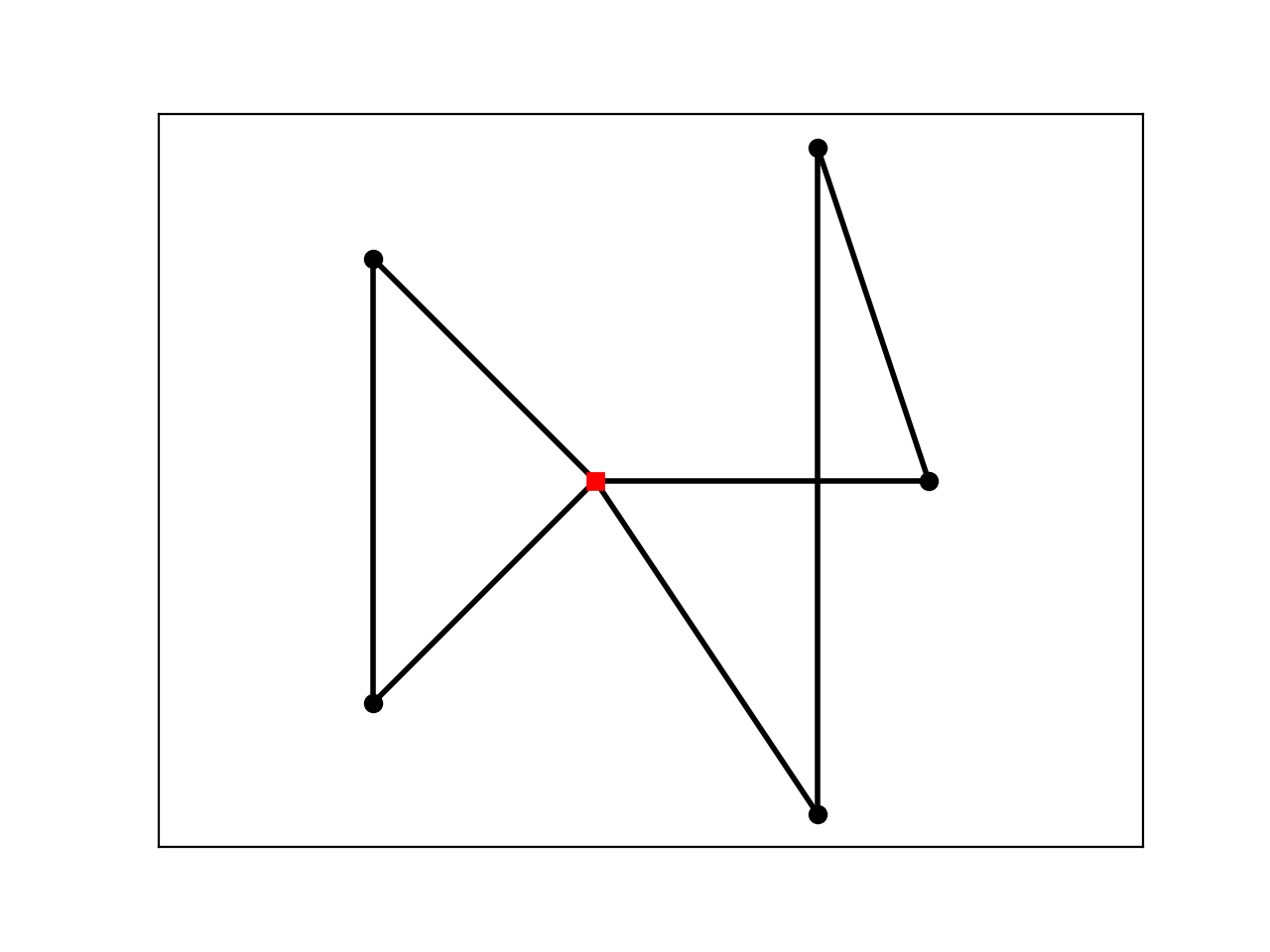}
        \caption{Optimal SCVRP routes using the weights $\theta_2$.}
        \label{fig:theta_2_routes}
    \end{subfigure}
    \begin{subfigure}[t]{.32\linewidth}
        \includegraphics[width = \linewidth]{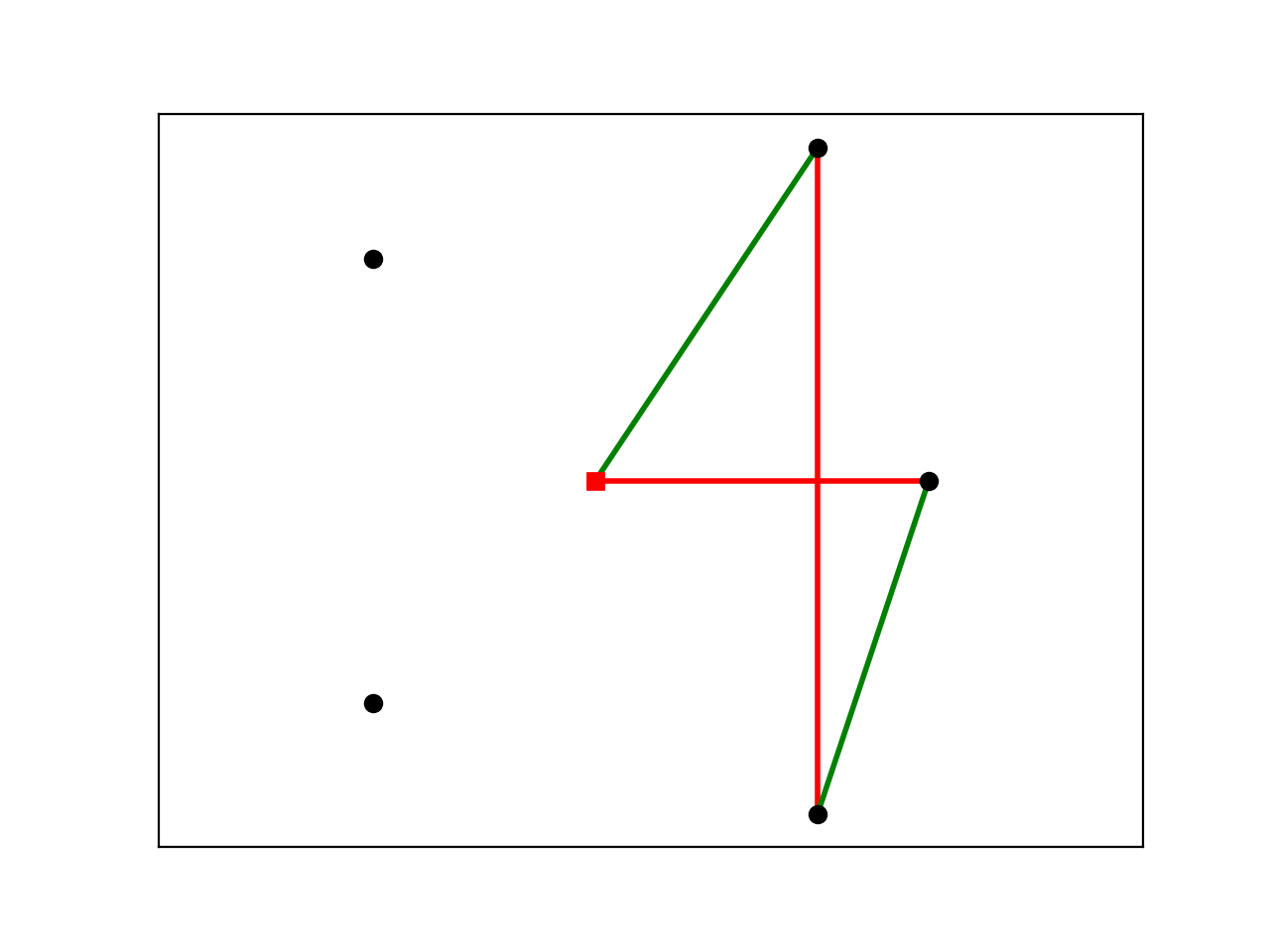}
        \caption{Difference between true optimal route (Figure \ref{fig:weights_and_routes}) and learned route (Figure \ref{fig:theta_2_routes}).}
        \label{fig:grad_2}
    \end{subfigure} \\
    \begin{subfigure}[t]{.32\linewidth}
        \includegraphics[width = \linewidth]{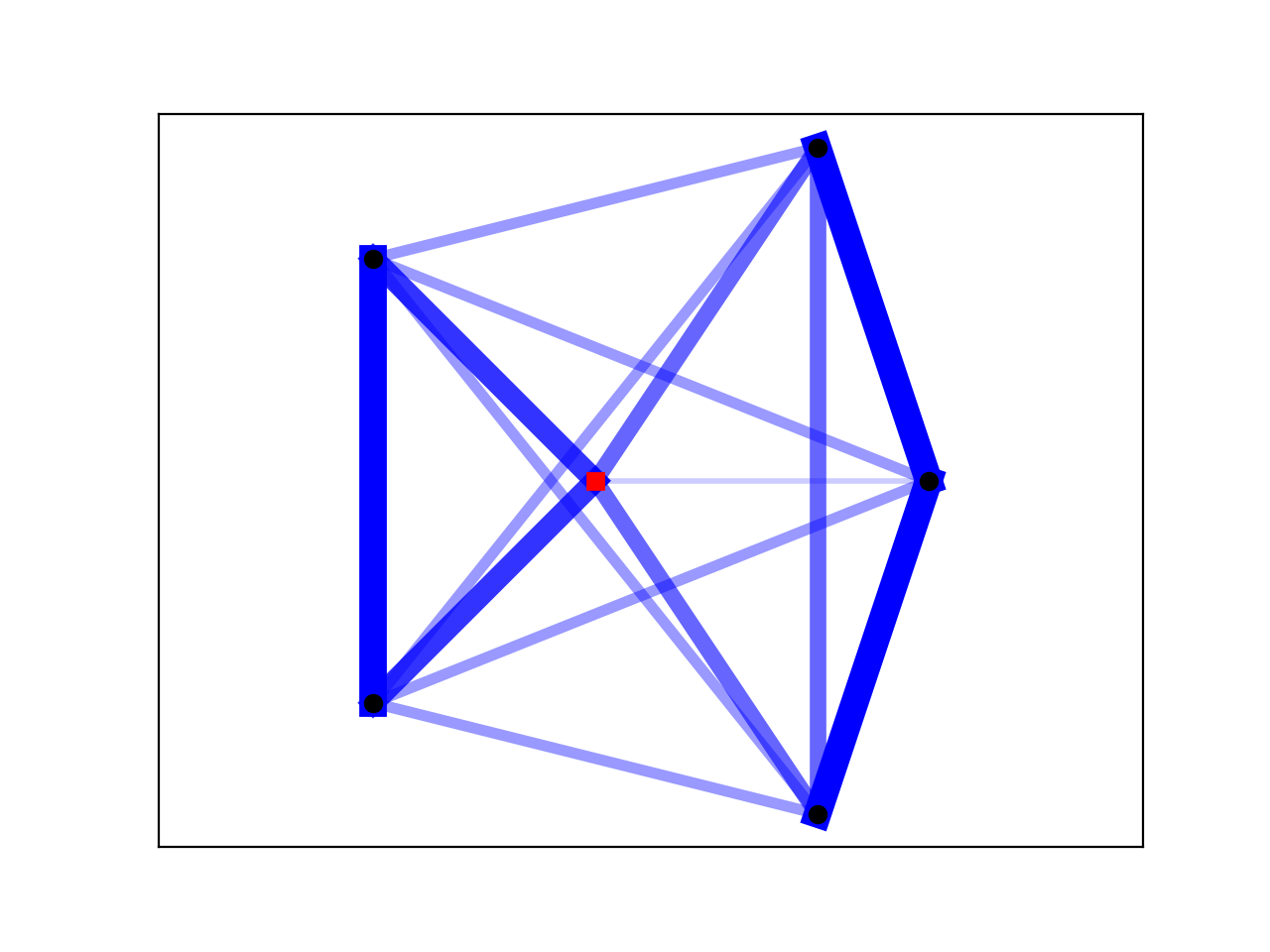}
        \caption{$\theta_3$.}
        \label{fig:theta_3}
    \end{subfigure}
    \begin{subfigure}[t]{.32\linewidth}
        \includegraphics[width = \linewidth]{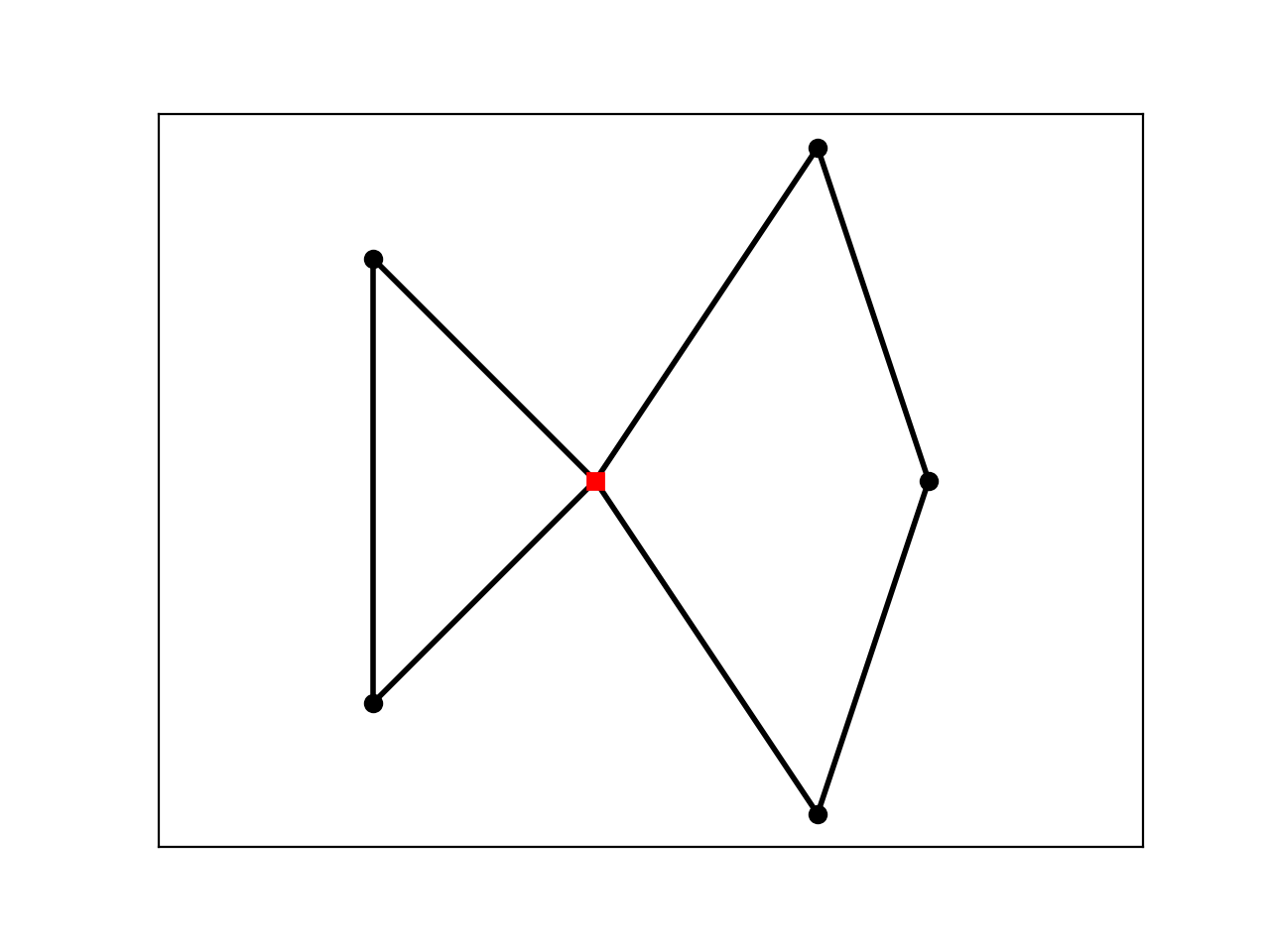}
        \caption{Optimal SCVRP routes using the weights $\theta_3$.}
        \label{fig:theta_3_routes}
    \end{subfigure}
    \begin{subfigure}[t]{.32\linewidth}
        \includegraphics[width = \linewidth]{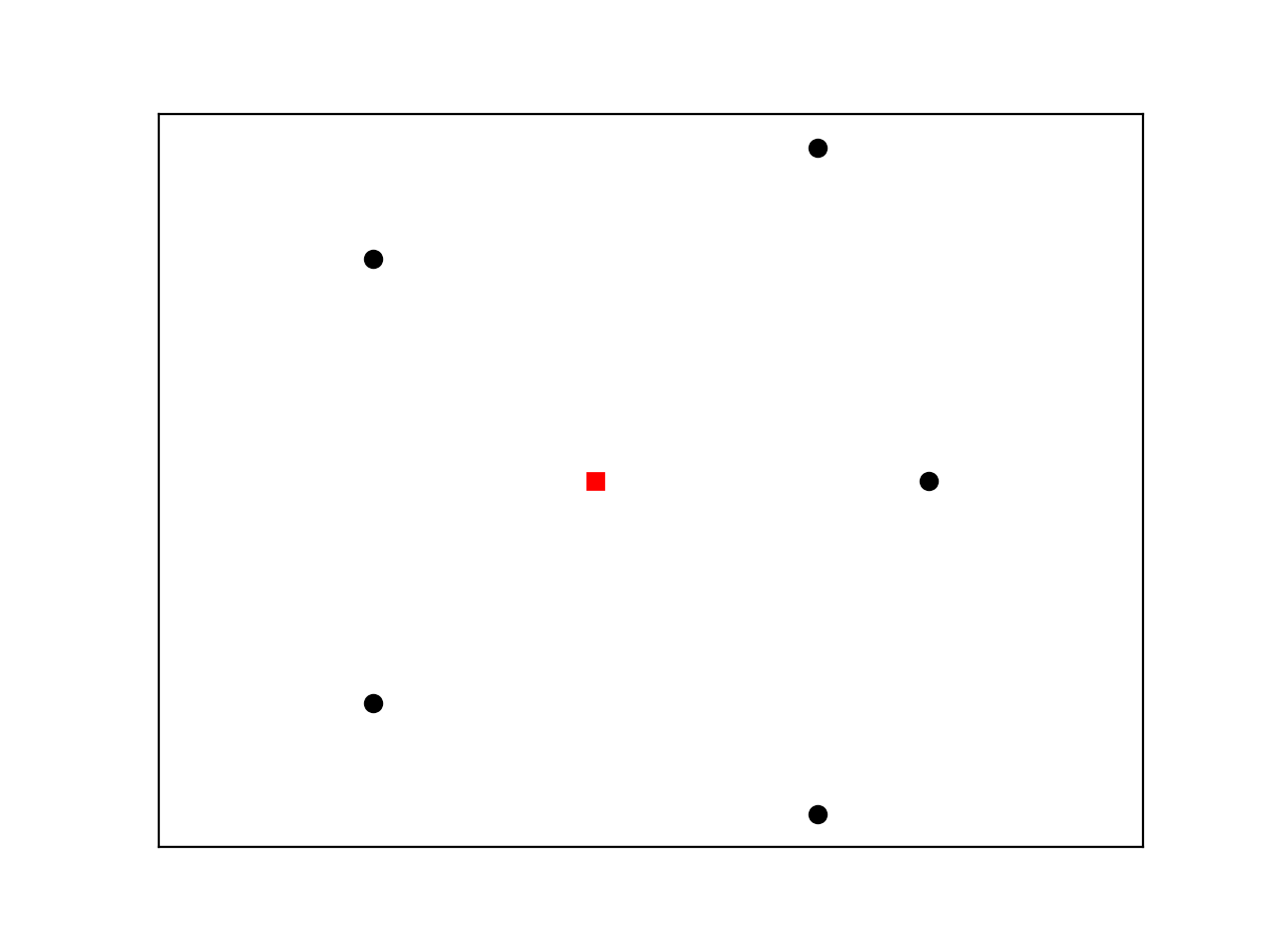}
        \caption{Difference between true optimal route (Figure \ref{fig:weights_and_routes}) and learned route (Figure \ref{fig:theta_3_routes}).}
        \label{fig:grad_3}
    \end{subfigure}    
    \caption{Two iterations of Algorithm \ref{alg:first_order}. The figures in the first column represent the learned weights, the figures in the second column are optimal SCVRP routes for the respective weights shown in the first column, and the third column represents the subgradient in line 6 of Algorithm \ref{alg:first_order}, where red (green) edges represent weights that should be increased (decreased).}
    \label{fig:OI_for_SCVRP}
\end{figure}

Next, we show how to use IO to learn edge weights that can be used to replicate the behavior of an expert, given a dataset of example routes. Consider the signal $\hat{s} \coloneqq D$, where $D$ is a set of demands of the customers, and the response $\hat{x} \in \{0, 1\}^{|E|}$, which is the vector with components $x_e$ encoding the optimal solution of the Problem \eqref{eq:SCVRP} for the signal $\hat{s}$. Defining the linear hypothesis function (i.e., $h(\hat{s},x) = 0$)
\begin{equation}
    \label{eq:SCVRP_hypothesis}
    \inner{\theta}{x} \coloneqq \sum_{e \in E} \theta_e x_e,
\end{equation}
and the constraint set
\begin{equation}
    \label{eq:SCVRP_set}
    \mathbb{X}(\hat{s}) \coloneqq
    \left\{x \in \{0,1\}^{|E|} :
    \begin{aligned}
        & \sum_{e \in \delta(i)} x_e = 2  \quad & \forall i \in V \setminus \{0\}, \\
        & \sum_{e \in \delta(0)} x_e = 2K, \\
        & \sum_{e \in \delta(S)} x_e \geq 2 r(S, D, c) \quad & \forall S \subset V \setminus \{0\}, \hspace{1mm} S \neq \emptyset
    \end{aligned}
    \right\},
\end{equation}
we can interpret the signal-response pair $(\hat{s}, \hat{x})$ as coming from an expert agent, which given the signal $\hat{s}$ of demands, solves the SCVRP to compute its response $\hat{x}$. Thus, to learn a cost function (i.e., learn a vector of edge weights) that replicates the SCVRP route $\hat{x}$, we can use Algorithm \ref{alg:first_order} to solve Problem \eqref{eq:loss_minimization} with hypothesis \eqref{eq:SCVRP_hypothesis} and constraint set \eqref{eq:SCVRP_set}. 

To illustrate how Algorithm \ref{alg:first_order} works to learn edge weights in routing problems on graphs, consider a simple SCVRP with $K=2$ vehicles, each with capacity $c=3$, and 5 customers, each customer $i$ with demand $d_i=1$. In this example, for simplicity, we use $w_e$ equal to the Euclidean distance between the customers, however, any other set of weights could be used instead. We create one training example using these weights. Figure \ref{fig:true_weights_routes} shows the location of the customers (black dots), the depot (red square), and the optimal SCVRP routes using weights $w_e$. Figure \ref{fig:true_weights} shows a representation of the weights of each edge of the graph, where the smaller the weight, the thicker and darker the edge. We use Algorithm \ref{alg:first_order} with exponentiated updates, $\eta_t = 0.0002$, and we initialize $\theta_1$ with the same weight for all edges. In Figure \ref{fig:OI_for_SCVRP}, we graphically show two iterations of the algorithm for this problem. In the first column, we show the evolution of the learned weights $\theta_t$. In the second column, we show optimal SCVRP routes computed using the weights in the first column (i.e., computed by solving the A-FOP in line 5 of Algorithm \ref{alg:first_order}), and in the third column, we show the difference between the optimal routes using the true weights (Figure \ref{fig:true_weights_routes}) and the optimal routes using the current learned weights (the route in the second column). The difference between these two routes is the subgradient computed in line 6 of Algorithm \ref{alg:first_order}, which is used to update the learned weights $\theta_t$. For the subgradient representation in the third column, red edges represent a negative subgradient (i.e., edges with weights that should be increased) and green edges represent a positive subgradient (i.e., edges with weights that should be decreased). This is the main intuition behind Algorithm \ref{alg:first_order}: at each iteration, we compare the route we want to replicate with the one we get with the current edge weights. Then, comparing which edges are used in these two routes, we either increase or decrease their respective weights, thus ``pushing'' the optimal route using the learned weights to be closer to the route we want to replicate. In the example shown in Figure \ref{fig:OI_for_SCVRP}, we can see that after two iterations of the Algorithm \ref{alg:first_order}, the optimal route using the learned weights coincides with the example route.

\subsection{IO for VRPTWs}
\label{sec:IO_for_VRPTW}

Consider the Vehicle Routing Problem with Time Windows (VRPTW)
\begin{equation}
\label{eq:VRPTW}
\begin{aligned}
    \min_{x_{ijk} \in \{0,1\}} \quad & \sum_{i=1}^n\sum_{j=1}^n\sum_{k=1}^K w_{ij} x_{ijk} \\
    \text{s.t.} \quad & x \in \mathbb{X}(\hat{s}),
\end{aligned} 
\end{equation}
where $n$ is the number of customers, $K$ is the maximum number of vehicles available, $x_{ijk}$ is a binary variable equal to $1$ if the edge from node $i$ to node $j$ is traversed by vehicle $k$ in the solution, and $0$ otherwise, and $w_{ijk}$ is the weight of the edge connecting node $i$ to node $j$. In the constraint set of program \eqref{eq:VRPTW}, $x$ is the vector containing the variables $x_{ijk}$, the signal $\hat{s}$ is defined to be the list of time windows (one for each customer) that need to be respected, and $\mathbb{X}(\hat{s})$ is the set of feasible solutions for the VRPTW for time windows in $\hat{s}$. Notice that the set $\mathbb{X}(\hat{s})$ may depend on other parameters of the problem, such as the service time of each customer, the demands of each customer, the travel time between customers, etc. However, we make the constraint set explicitly dependent only on the time windows since this is the only external parameter that will change in this example. More details on the different formulations for the constraint set of VRPTWs can be found in \cite{toth2002vehicle}.

Next, we show how one can model the VRPTW into our IO framework. Consider the dataset $\{(\hat{s}^{[i]}, \hat{x}^{[i]})\}_{i=1}^N$, where the signal $\hat{s}^{[i]}$ is the list of time windows that need to be respected and the response $\hat{x}^{[i]} \in \{0,1\}^{n^2K}$ is the respective optimal VRPTW routes (i.e., a vector with components $x_{ijk}$). Defining the linear hypothesis function
\begin{equation}
    \label{eq:VRPTW_hypothesis}
    \inner{\theta}{x} \coloneqq \sum_{i=1}^n\sum_{j=1}^n\sum_{k=1}^K \theta_{ij} x_{ijk},
\end{equation}
we can interpret this dataset as coming from an expert agent, which given the signal $\hat{s}^{[i]}$, solves a VRPTW to compute its response $\hat{x}^{[i]}$. Thus, to learn a cost function (i.e., learn a vector of edge weights) that replicates the VRPTW route $\hat{x}$, we can use Algorithm \ref{alg:first_order} to solve Problem \eqref{eq:loss_minimization} with hypothesis \eqref{eq:VRPTW_hypothesis} and the constraint set of \eqref{eq:VRPTW}.

To illustrate this formulation, we will use a VRPTW scenario generated using data from the \textit{EURO Meets NeurIPS 2022 Vehicle Routing Competition} \cite{ortec2022euro}. The VRPTWs considered in this competition are real-world instances provided by the company ORTEC. To generate the training data to test our IO formulation, we pick one instance from the competition, which corresponds to a relatively large VRPTW with $n = 200$ customers and $K = 15$ available vehicles. Originally, each customer in this VRPTW instance had fixed time windows. However, to generate an IO dataset, we shuffled the original time windows among the 200 customers and computed the optimal VRPTW routes for each of these new instances. Thus, we generate a dataset $\{(\hat{s}^{[i]}, \hat{x}^{[i]})\}_{i=1}^N$, where the signal $\hat{s}^{[i]}$ is a random assignment of time-windows to customers, and the response $\hat{x}^{[i]}$ is the respective optimal VRPTW solution. Using the state-of-the-art solver PyVRP \cite{wouda2024PyVRP}, we generated $N = 50$ training and test instances. All these instances have the same true edge weights $w_{ij}$, which corresponds to the non-euclidean real-world road driving time from customer $i$ to customer $j$. Thus, our IO goal is to learn a set of weights $\theta_{ij}$ that replicate the routes using $w_{ij}$ as well as possible, given the provided dataset of signal-response training examples. We learn the weights using the training dataset and evaluate its performance using a test dataset. Moreover, we report the average performance value for 5 randomly generated training/test datasets, as well as the 5th and 95th percentile bounds. We test three approaches to solve the IO problem, where we set the initial weights in $\theta_1$ equal to the Euclidean distance between customers $i$ and $j$.
\begin{itemize}
    \item \textbf{Cutting plane:} We use the cutting plane algorithm from \cite{wang2009cutting} to solve
    \begin{equation*}
        \begin{aligned}
        \min_{\theta \geq 0} \quad & \|\theta - \theta_1 \|_1 \\
        \text{s.t.} \quad & \hat{x}_i \in \text{FOP}(\theta, \hat{s}_i) \quad \forall i \in [N],
        \end{aligned}
    \end{equation*}
    which is the multi-point IO formulation proposed in \cite{bodur2022inverse}.
    \item \textbf{SAMD:} We use the SAMD algorithm from \cite{zattoniscroccaro2023learning} to solve \eqref{eq:loss_minimization}, with exponentiated updates and $\eta_t = 0.3/t$.
    \item \textbf{Algorithm 1:} We use Algorithm \ref{alg:first_order} to solve \eqref{eq:loss_minimization}, with exponentiated updates and $\eta_t = 0.3/t$. For this example, the difference between the SAMD algorithm from \cite{zattoniscroccaro2023learning} and Algorithm \ref{alg:first_order} is that the former uses uniform stochastic sampling, while the latter uses the reshuffled sampling strategy.
\end{itemize}

\begin{figure}
\centering
\captionsetup[subfigure]{width=0.96\linewidth}%
    \begin{subfigure}[t]{0.32\linewidth}
        \includegraphics[width = \linewidth]{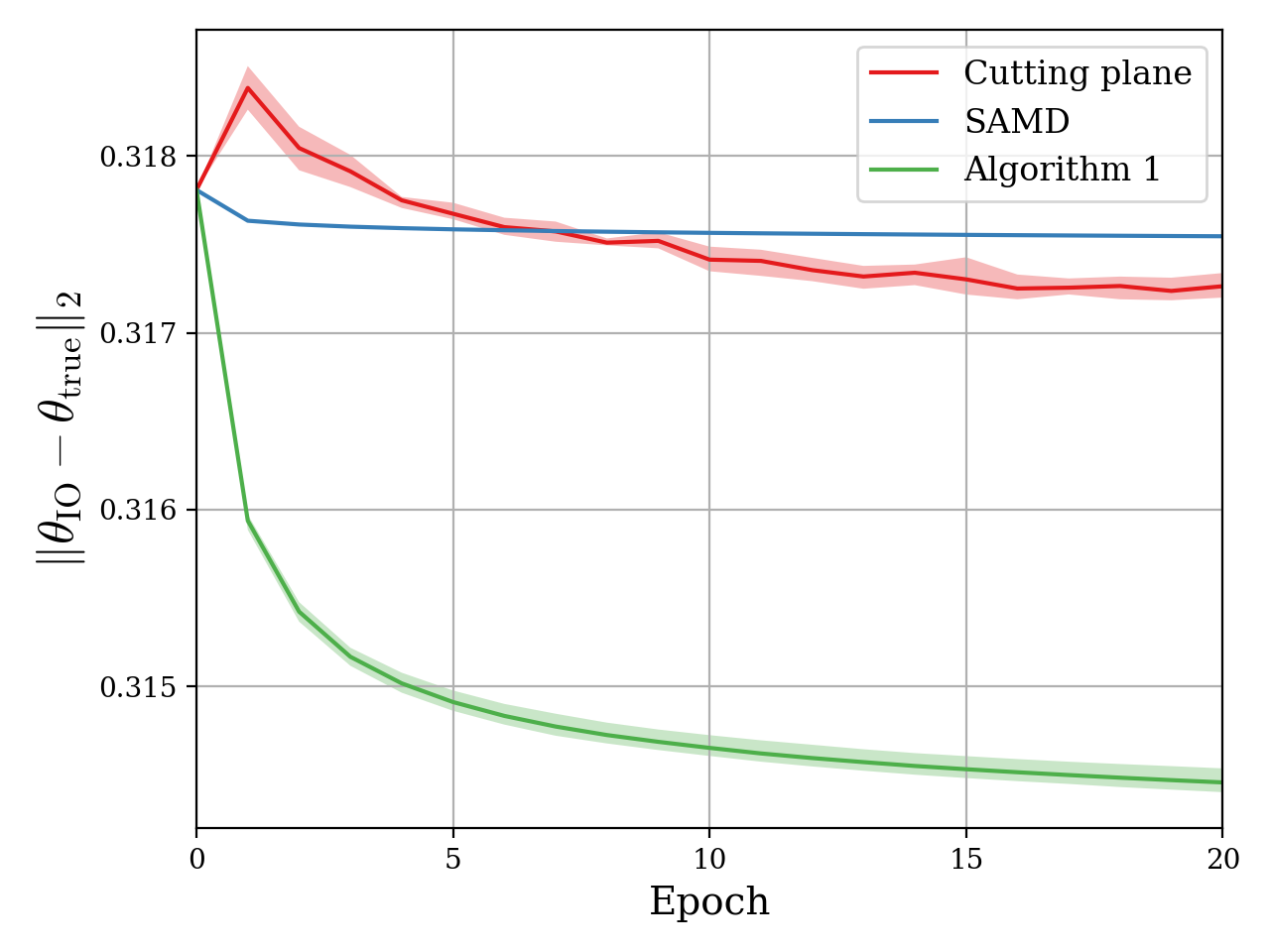}
        \caption{Difference between the true weights ($\theta_{\text{true}}$) and the ones learned using IO ($\theta_{\text{IO}}$).}
        \label{fig:vrptw_theta_diff}
    \end{subfigure}
    \begin{subfigure}[t]{0.32\linewidth}
        \includegraphics[width = \linewidth]{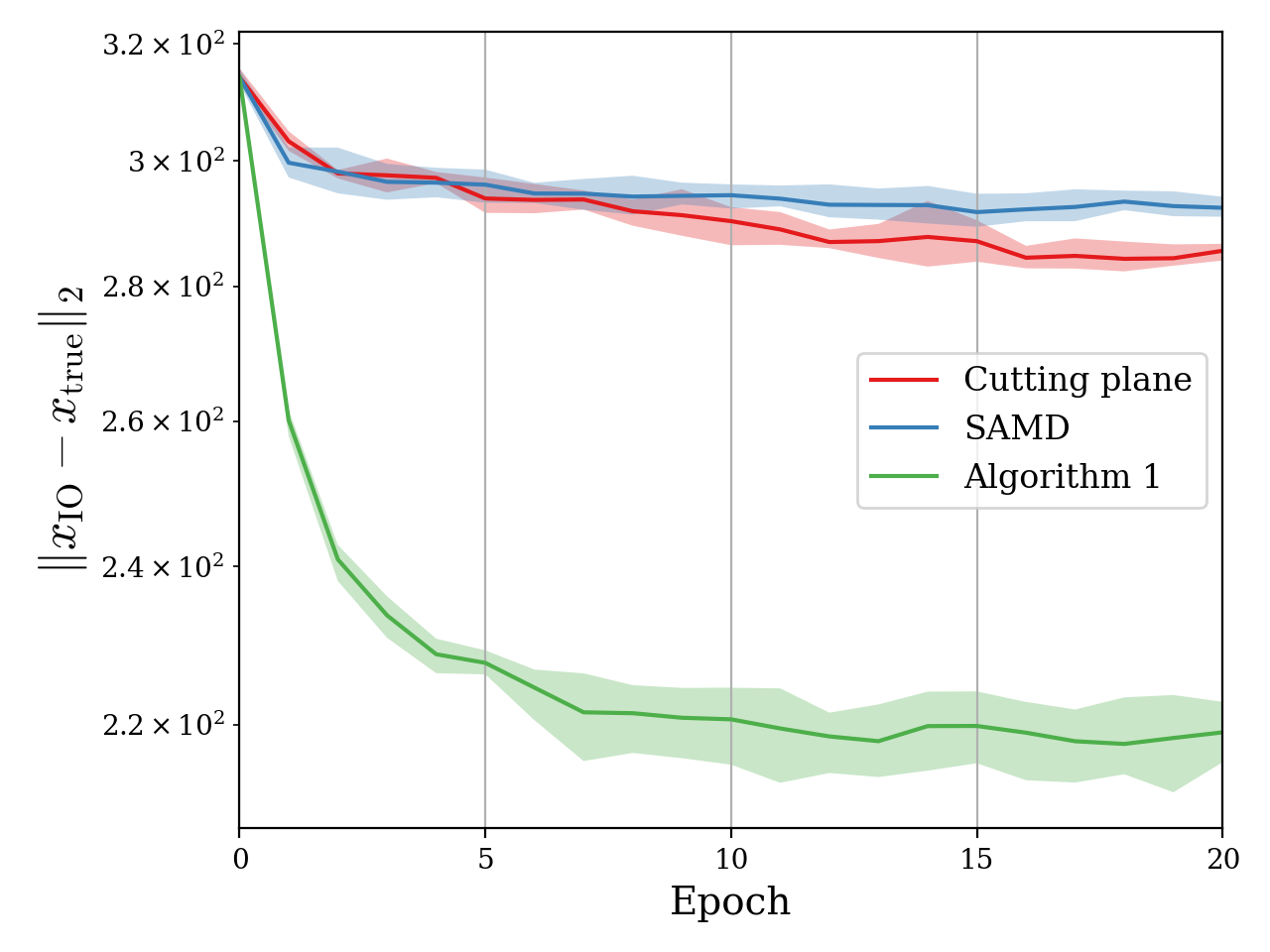}
        \caption{Average error between the routes generated by $\theta_{\text{true}}$ and $\theta_{\text{IO}}$.}
        \label{fig:vrptw_x_diff_out}
    \end{subfigure}
    \begin{subfigure}[t]{0.32\linewidth}
        \includegraphics[width = \linewidth]{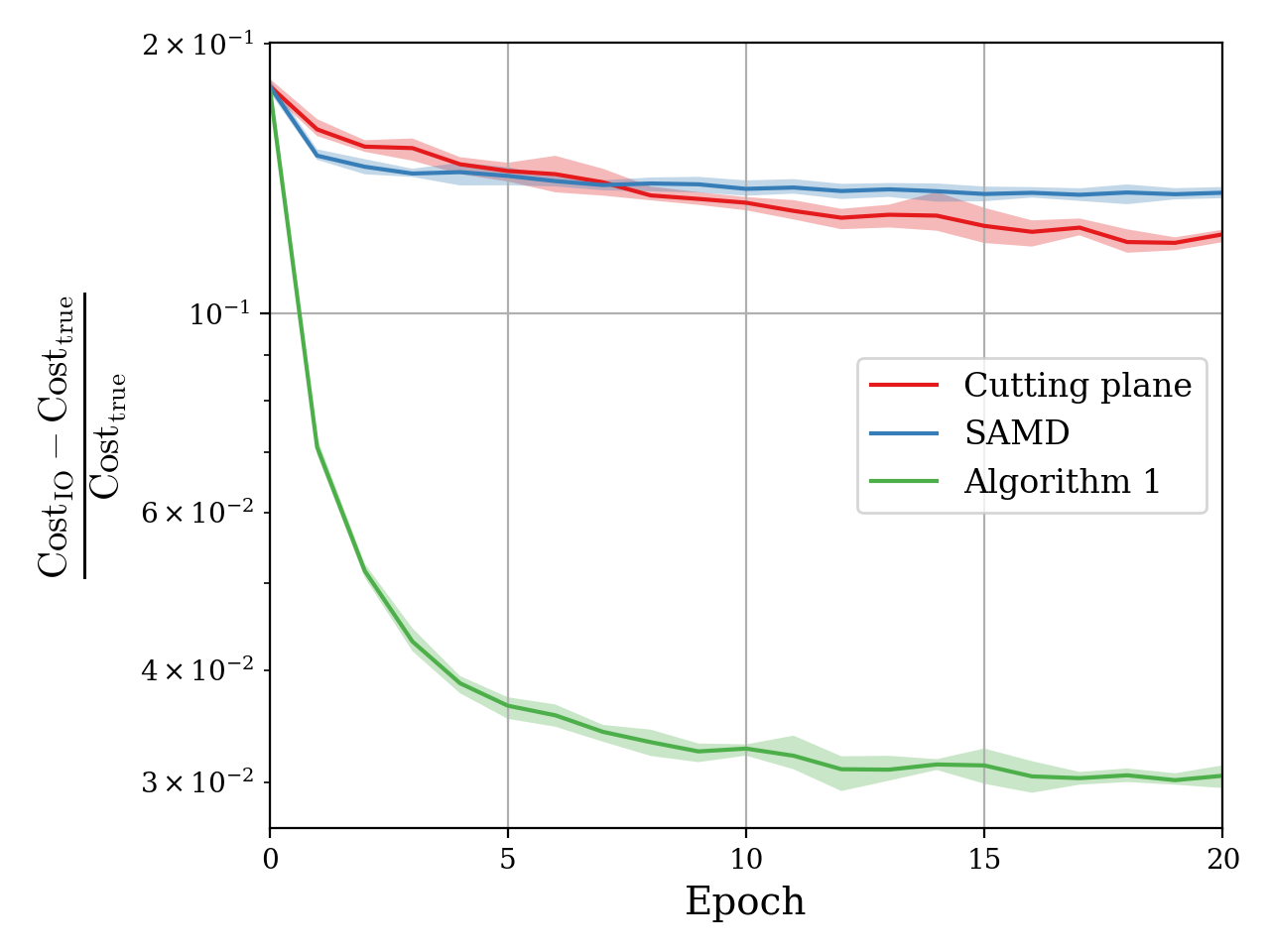}
        \caption{Average normalized cost difference between the routes generated by $\theta_{\text{true}}$ and $\theta_{\text{IO}}$.}
        \label{fig:vrptw_cost_diff_out}
    \end{subfigure}
\caption{Results for the VRPTW scenario.}
\label{fig:vrptw_out}
\end{figure}

Our experiments are reproducible, and the underlying source code is available at \cite{zattoniscroccaro2023invopt}. Figure \ref{fig:vrptw_out} shows the results of this experiment. For all the plots, the x-axis refers to the epoch $t \in [1, T]$, which consists of $N$ iterations of the method used to solve the problem. Figure \ref{fig:vrptw_theta_diff} shows the normalized difference between the vector of weights returned by the IO approach (which we name $\theta_{\text{IO}}$) and the vector of weights used to generate the data (which we name $\theta_{\text{true}}$). Figure \ref{fig:vrptw_x_diff_out} shows the average difference between the optimal routes using $\theta_{\text{IO}}$ (which we name $x_{\text{IO}}$) and the routes from the test dataset (which we name $x_{\text{true}}$). Figure \ref{fig:vrptw_cost_diff_out} shows the normalized difference between the cost of the expert decisions and the cost of the decisions using $\theta_{\text{IO}}$. More precisely, we define $\text{Cost}_{\text{IO}} \coloneqq \sum_{i=1}^N \langle \theta_{\text{true}}, x_\text{IO}^{[i]} \rangle$ and $\text{Cost}_{\text{true}} \coloneqq \sum_{i=1}^N \langle \theta_{\text{true}}, x_\text{true}^{[i]} \rangle$ and compare the relative difference between them. Notice that this difference will always be nonnegative by the optimality of $x_\text{true}^{[i]}$. From the results of this experiment, we can see that Algorithm \ref{alg:first_order} outperforms the other approaches (i.e., the cutting plane and SAMD) by a relatively large margin, which shows the efficacy of our proposed reshuffled sampling strategy (i.e., Algorithm \ref{alg:first_order}) for this example.

\subsection{IO for TSPs}
\label{sec:IO_for_TSP}

Let $\mathcal{G} = (V, E, W)$ be a complete edge-weighted directed graph, with node set $V$, directed edges $E$, and edge weights $W$. Next, given $\hat{s} \subset V$ (i.e., a subset of the nodes of $\mathcal{G}$), we define the \textit{Restricted Traveling Salesperson Problem} (R-TSP) as
\begin{equation}
\label{eq:R-TSP}
\begin{aligned}
    \min_{x_{ij}} \quad & \sum_{i \in V} \sum_{j \in V} w_{ij} x_{ij}, \\
    \text{s.t.} \quad & \sum_{j \in \hat{s}} x_{ij} = \sum_{j \in \hat{s}} x_{ji} = 1  \quad & \forall i \in \hat{s} \\
    & \sum_{i \in Q} \sum_{j \in Q} x_{ij}  \leq |Q| -1  \quad & \forall Q \subset \hat{s}, \hspace{1mm} Q \neq \emptyset, \hspace{1mm} \bar{Q} \neq \emptyset \\
    &x_{ij} \in \{0,1\} & \forall (i,j) \in \hat{s} \times \hat{s} \\
    &x_{ij} = 0 & \forall (i,j) \notin \hat{s} \times \hat{s},
    \end{aligned} 
\end{equation}
where $x_{ij}$ is a binary variable equal to $1$ if the edge from node $i$ to node $j$ is used in the solution, and $0$ otherwise, and $w_{ij}$ is the weight of the edge connecting node $i$ to node $j$. Problem \eqref{eq:R-TSP} is based on the standard formulation of a TSP as a binary optimization problem \cite{dantzig1954solution}. The only difference to a standard TSP is that instead of being required to visit all nodes of the graph, for an R-TSP we compute the optimal tour over a subset $\hat{s}$ of the nodes $V$. Notice that the standard TSP can be interpreted as an R-TSP, for the special case when $\hat{s} = V$. In practice, any TSP solver can be used to solve an R-TSP by simply ignoring all nodes of the graph that are not required to be visited.

Next, we show how to use IO to learn edge weights that can be used to replicate the behavior of an expert, given a set of example routes. Consider the dataset $\{(\hat{s}^{[i]}, \hat{x}^{[i]})\}_{i=1}^N$, where the signal $\hat{s}^{[i]} \in V$ is a set of nodes required to be visited and the response $\hat{x}^{[i]} \in \{0,1\}^{|V|^2}$ is the respective optimal R-TSP tour (i.e., a vector with components $x_{ij}$ for $(i,j) \in V \times V$). Defining the affine hypothesis function
\begin{equation}
    \label{eq:TSP_hypothesis}
    \inner{\theta}{x} + h(\hat{s}, x) \coloneqq \sum_{i \in V} \sum_{j \in V} \big(\theta_{ij} + M_{ij} \big) x_{ij},
\end{equation}
and the constraint set
\begin{equation}
    \label{eq:TSP_set}
    \mathbb{X}(\hat{s}) \coloneqq
    \left\{x \in \{0,1\}^{|V|^2} :
    \begin{aligned}
        & \sum_{j \in \hat{s}} x_{ij} = 1, & \forall i \in \hat{s} \\
        & \sum_{i \in \hat{s}} x_{ij} = 1, & \forall j \in \hat{s} \\
        & \sum_{i \in Q} \sum_{j \in Q} x_{ij} \leq |Q| -1, & \forall Q \subset \hat{s}, Q \neq \emptyset, \bar{Q} \neq \emptyset \\
        & x_{ij} = 0 & \forall (i,j) \notin \hat{s} \times \hat{s}
    \end{aligned}
    \right\},
\end{equation}
we can interpret this dataset as coming from an expert agent, which given the signal $\hat{s}^{[i]}$, solves an R-TSP to compute its response $\hat{x}^{[i]}$. For the hypothesis function, the term $M_{ij}$ can be used as a penalization term to enforce some kind of expected behavior to the model, e.g., by adding penalizations to some edges of the graph. Figure \ref{fig:IO_TSP} illustrates a signal and expert response for an R-TSP. Thus, to learn a cost function (i.e., learn a vector of edge weights) that replicates (or approximates as well as possible) the example routes in the dataset, we can use Algorithm \ref{alg:first_order} to solve Problem \eqref{eq:loss_minimization} with hypothesis \eqref{eq:TSP_hypothesis} and constraint set \eqref{eq:TSP_set}. This formulation will serve as the basis of our IO approach to tackle the Amazon Challenge.

\begin{figure}
\centering
\captionsetup[subfigure]{width=0.96\linewidth}%
    \begin{subfigure}[t]{0.28\linewidth}
        \includegraphics[width = \linewidth]{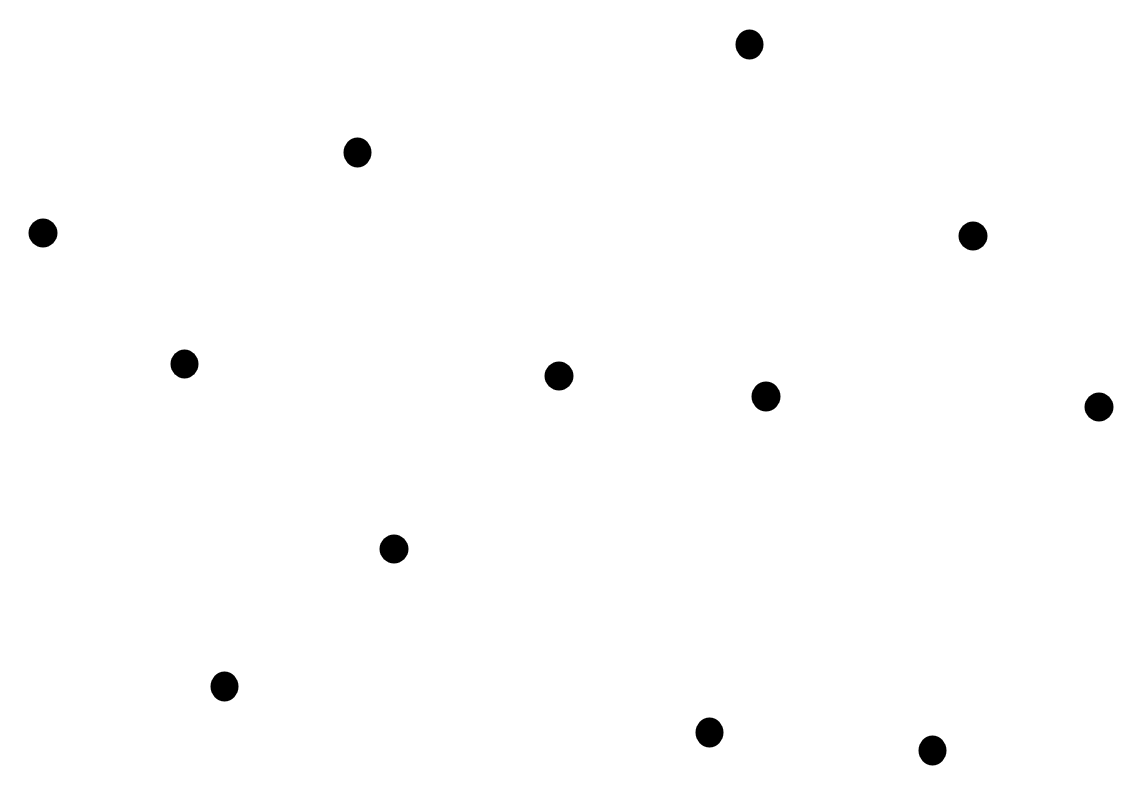}
        \caption{Nodes of a graph $\mathcal{G}$.}
        \label{fig:graphG}
    \end{subfigure}
    \hfill
    \begin{subfigure}[t]{0.28\linewidth}
        \includegraphics[width = \linewidth]{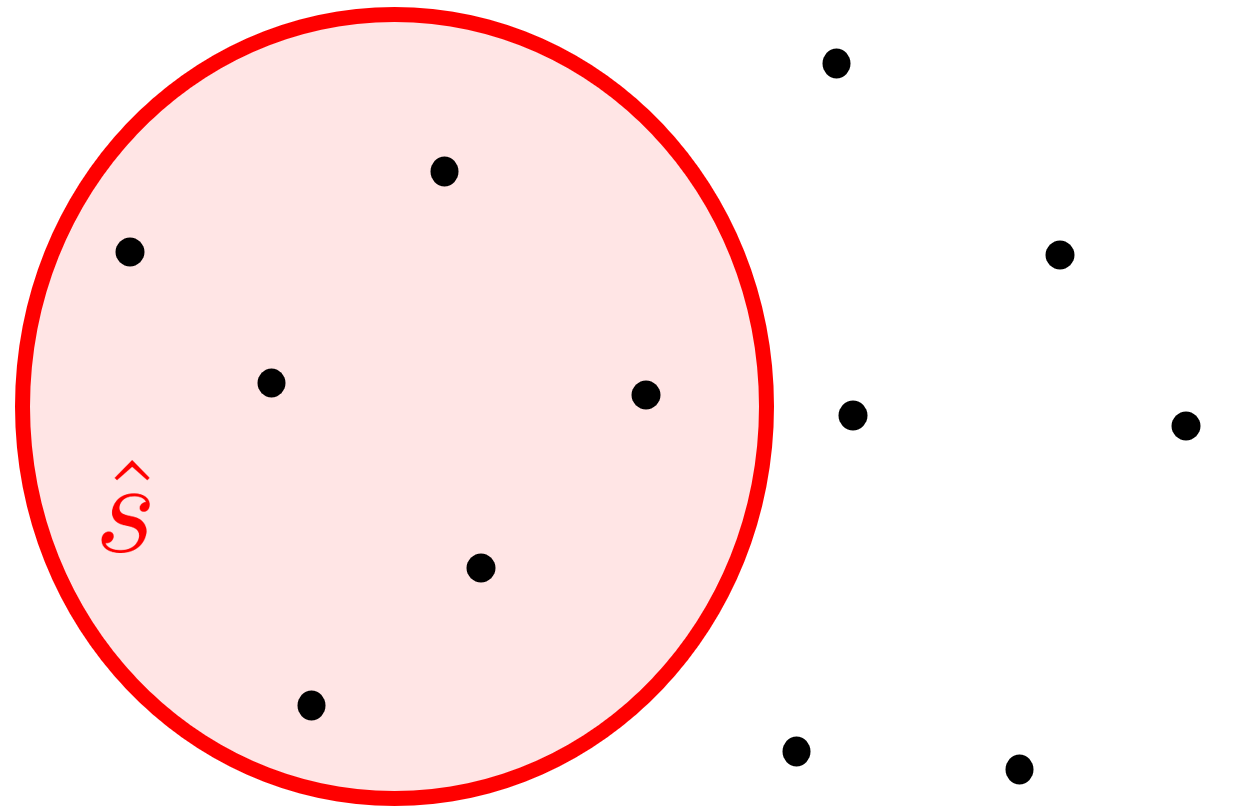}
        \caption{Signal.}
        \label{fig:signal}
    \end{subfigure}
    \hfill
    \begin{subfigure}[t]{0.28\linewidth}
        \includegraphics[width = \linewidth]{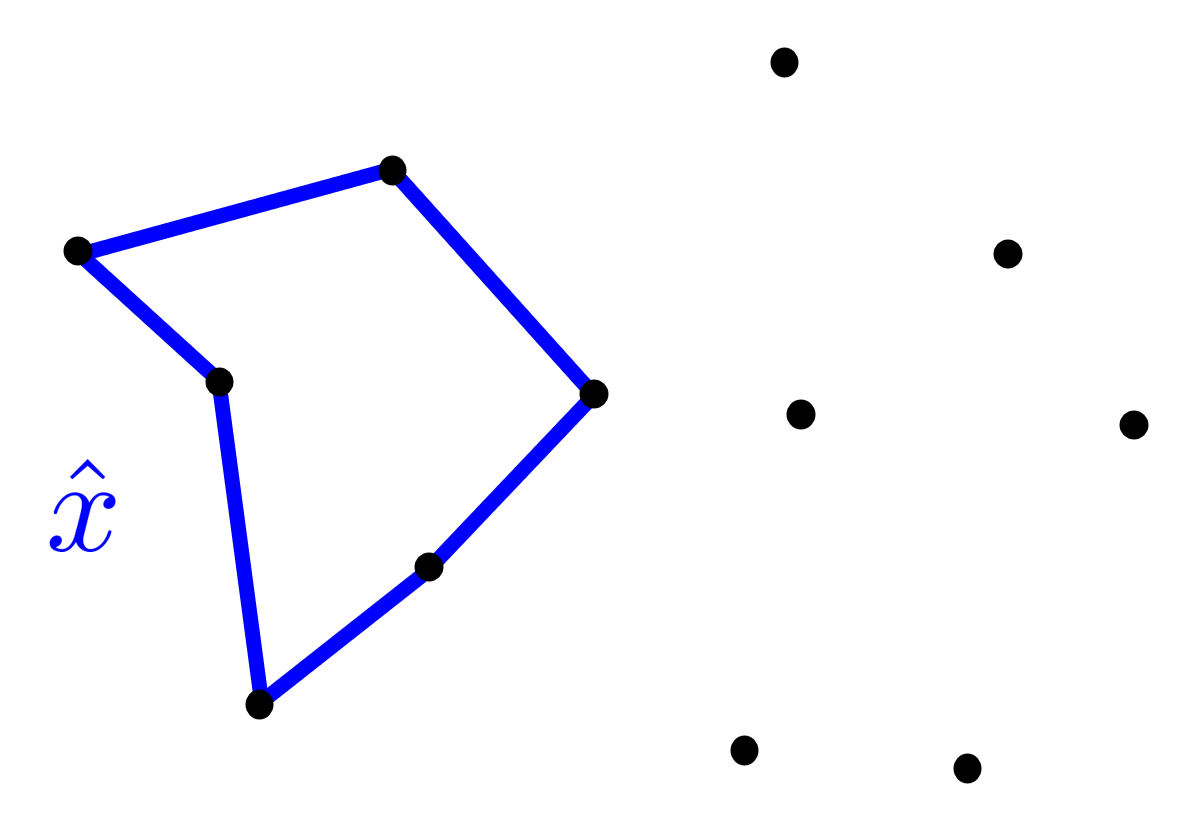}
        \caption{Expert response.}
        \label{fig:action}
    \end{subfigure}
\caption{Illustration of signal and expert response for an R-TSP.}
\label{fig:IO_TSP}
\end{figure}

We conclude this section with some general comments about our IO approach. First, our IO approach does not require the dataset $\{(\hat{s}^{[i]}, \hat{x}^{[i]})\}_{i=1}^N$ to be consistent with a single cost function (i.e. a single set of edge weights), which is to be expected in any realistic setting, due to model uncertainty, noisy measurements or bounded rationality \cite{mohajerin2018data}. Also, we showed how to use our IO approach for SCVRPs, VRPTWs, and R-TSP scenarios, but we emphasize that the methodology developed in this section could be easily adapted to different kinds of routing problems. For instance, if the problem was a VRP backhauls, or pickup and delivery locations, we could easily account for these characteristics, for example, by changing the constraint set $\mathbb{X}(\hat{s})$ of our IO model \cite{toth2002vehicle}, or in other words, by modifying the problem we assumed the expert agent is solving to generate its response. Notice that in any case, the methodology developed in sections \ref{sec:hypothesis}, \ref{sec:loss_func}, and \ref{sec:algorithm} would not change, which highlights the generality and flexibility of our IO approach. As a final comment, we mention that our approach can easily be adapted to the scenario where new signal-response examples arrive in an \textit{online} fashion. That is, instead of learning from an offline dataset of examples, we gradually update the edge weights (i.e., $\theta_t$) with examples that arrive online, similar to \cite{barmann2017emulating}. This can be done straightforwardly by adapting Algorithm \ref{alg:first_order} to use examples that arrive online in the same way it uses the signal-response pairs $(\hat{s}^{[\pi_i]}, \hat{x}^{[\pi_i]})$.


\section{Amazon Challenge}
\label{sec:challenge}

In this section, we describe the Amazon Challenge, which we use as a real-world application to assess our IO approach. A detailed description of the data provided for the challenge can be found in \cite{merchan20222021}. In summary, Amazon released two datasets for this challenge: a training dataset and a test dataset. The training dataset consists of 6112 historical routes driven by experienced drivers. This dataset is composed of routes performed in the metropolitan areas of Seattle, Los Angeles, Austin, Chicago, and Boston, and each route is characterized by several features. Figure \ref{fig:dataset_description} shows a high-level description of the features available for each example route. Each of these routes starts at a depot, visits a collection of drop-off stops assigned to the driver in advance, and ends at the same depot. Thus, each route can be interpreted as an R-TSP route. Figure \ref{fig:DBO1_all_routes} shows 8 example routes leaving from a depot in Boston, where different colors represent different routes. Each stop in every route was given a Zone ID, which is a unique identifier denoting the geographical planning area into which the stop falls, and is devised internally by Amazon \cite{merchan20222021}. Some stops in the dataset are not given a Zone ID, so for these stops, we assign them the Zone ID of the closest zone (in terms of Euclidean distance). Turns out, this predefined zoning of the stops is a key piece of information about the Amazon Challenge. This will be discussed in detail in the subsequent sections of this paper.

\begin{figure}
    \centering
    \includegraphics[width = 1\linewidth]{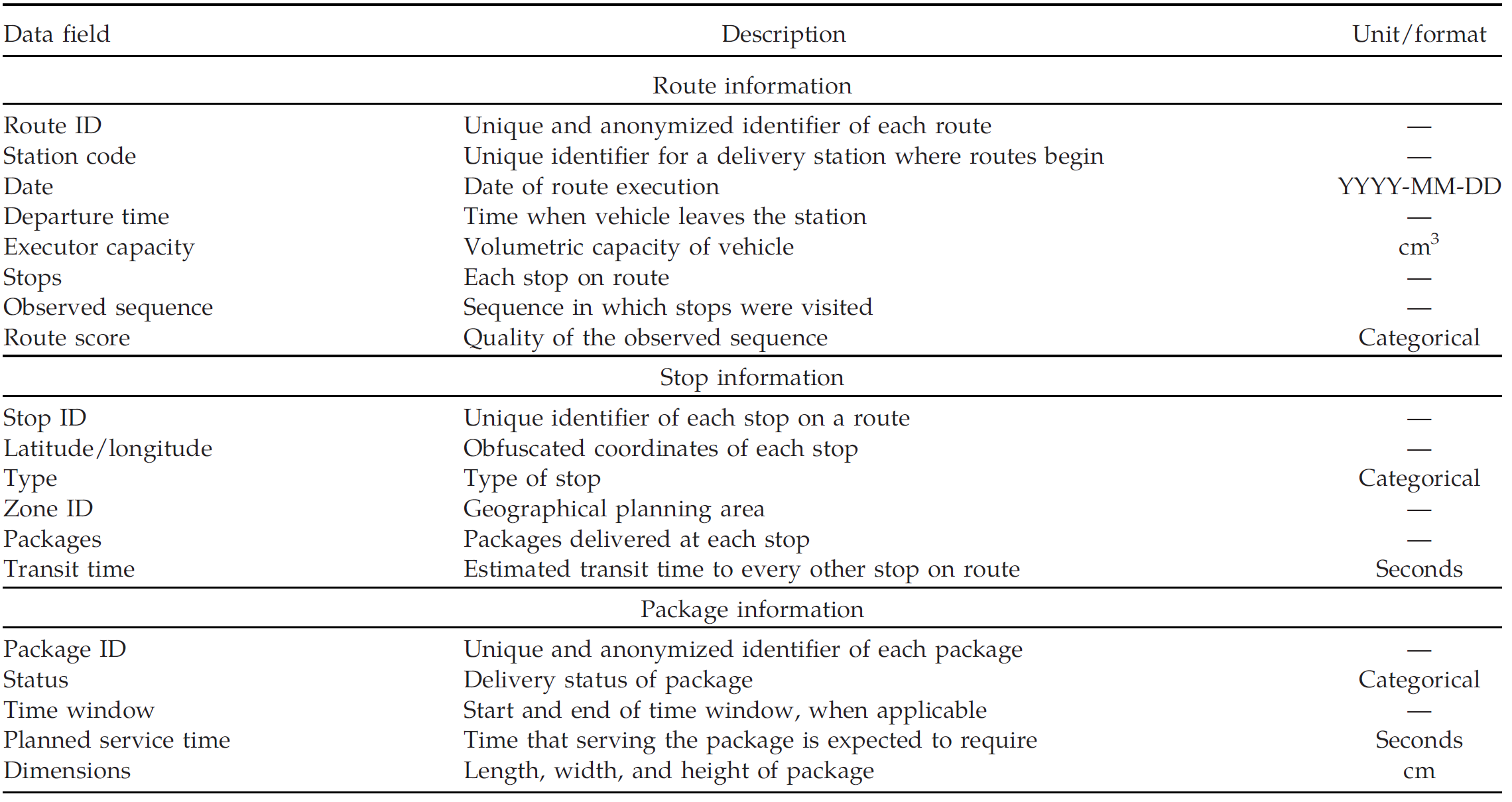}
    \caption{High-level description of data fields provided in the Amazon Challenge data set \cite{merchan20222021}.}
    \label{fig:dataset_description}
\end{figure}

\begin{figure}
    \centering
    \includegraphics[width = 0.5\linewidth]{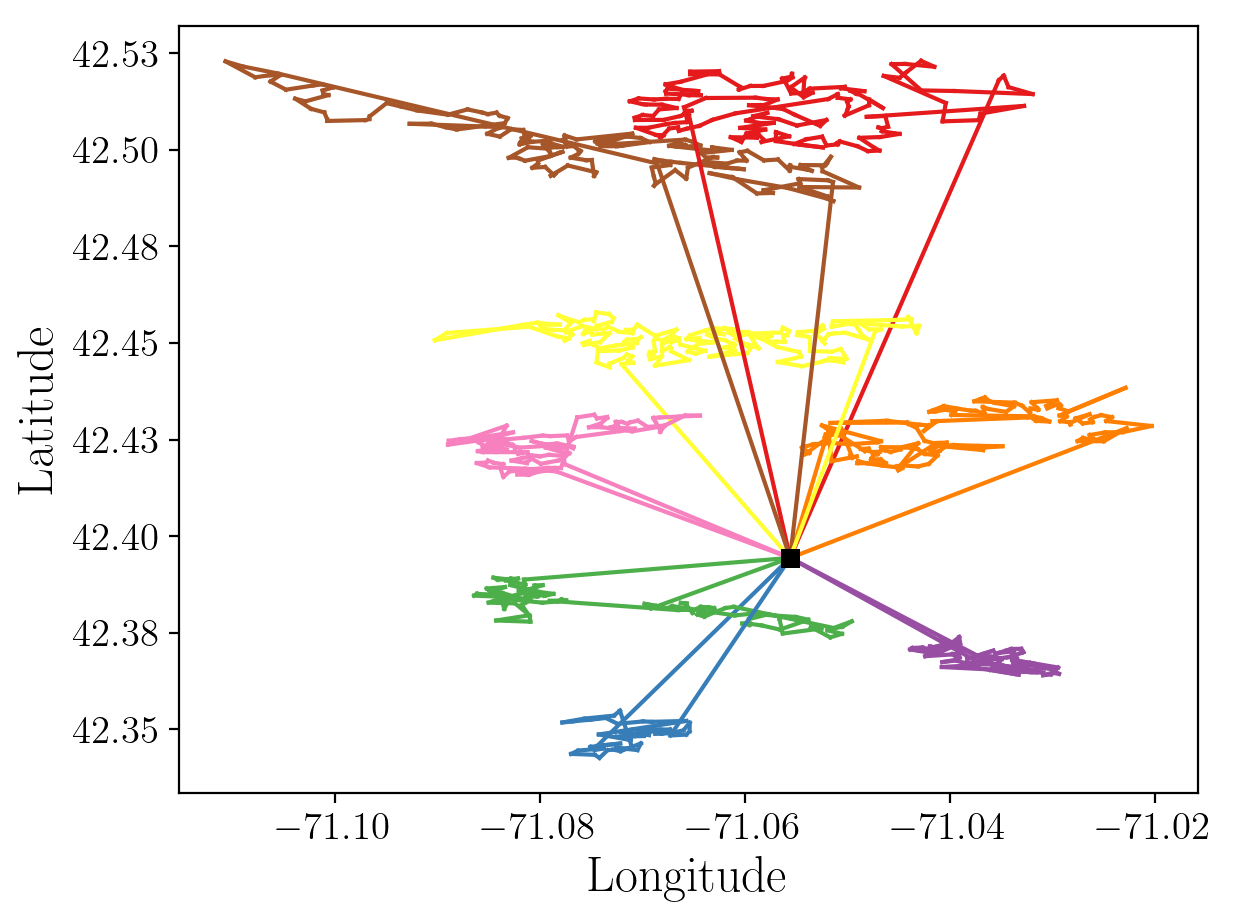}
    \caption{Example routes from depot DBO1 in Boston, where different colors represent different routes.}
    \label{fig:DBO1_all_routes}
\end{figure}


As previously mentioned, the goal of the challenge was to incorporate the preferences of experienced drivers into the routing of last-mile delivery vehicles. Thus, rather than coming up with TSP strategies that minimize time or distance given a set of stops to be visited, the goal of the challenge was to learn from historical data how to route like the expert drivers. To this end, a test dataset consisting of 3072 routes was also made available to evaluate the proposed approaches. To compare the routes from expert human drivers to the routes generated by the models submitted to the challenge, Amazon devised a scoring metric that computes the similarity between two routes, where the lower the score, the more similar the routes. In particular, if $A$ is the historically realized sequence of deliveries, sequence $B$ is the sequence of deliveries generated by a model, its score is defined as follows:
\begin{equation}
    \label{eq:amazon_score}
    \text{score}(A, B) = \frac{SD(A,B) \cdot {ERP}_{norm}(A,B)}{{ERP}_e(A,B)},
\end{equation}
where sequence $SD$ denotes the Sequence Deviation of $B$ with respect to $A$, $ERP_{norm}$ denotes the Edit Distance with Real Penalty applied to sequences $A$ and $B$ with normalized travel times, and ${ERP}_e$ denotes the number of edits prescribed by the $ERP$ algorithm on sequence $B$ with respect to $A$. If edit distance with real penalty prescribes 0 edits, then the above formula is replaced by the sequence deviation, multiplied by 0. Thus, the Amazon score combines a similarity measure that takes into account only the sequence of stops in the routes (i.e., $SD$) with a similarity measure that also takes the travel times between stops into account (i.e., $ERP$). The details of the score computation can be found at \cite{amazon2021scoring}. Notice that, instead of using our tailored loss function \eqref{eq:loss_function}, one could use the scoring function \eqref{eq:amazon_score} directly to learn the routing model, which makes intuitive sense since minimizing this score is the actual goal of the Amazon Challenge. However, the resulting IO problem would be an intractable bi-level optimization problem, similar to the case when using the so-called \textit{predictablity loss} for IO \cite{aswani2018inverse}. This issue highlights one of the big advantages of using our tailored loss function \eqref{eq:loss_function}: the resulting optimization problem is convex and subgradients of the loss function can be computed in closed form, thus, making the problem amendable to be solved using efficient first-order methods, such as Algorithm \ref{alg:first_order}.

In summary, a dataset of 6112 historical routes from expert human drivers was made available for the Amazon Challenge. Using this dataset, the goal is to come up with routing methods that replicate the way human drivers route vehicles. To evaluate the proposed approaches, Amazon used a test dataset consisting of 3072 unseen examples. In order to compare how similar the routes from this dataset are to the ones computed by the submitted approaches, a similarity score was devised. The final score is the average score over the 3072 test instances. A summary of the scores of the top 20 submissions to the Amazon Challenge can be found at \cite{amazon2021last}. Since each historical route in the dataset of the challenge refers to a driver's route that starts at a depot, visits a predefined set of customers, and then returns to the depot, the expert human routes from the Amazon Challenge can be interpreted as solutions to R-TSPs, and we can use the IO approach to tackle the Amazon Challenge. In other words, we can use IO to estimate the costs they assign to the street segments connecting stops. Ultimately, this will allow us to learn the drivers' preferences, and replicate their behavior when faced with new requests for stops to be visited.

\subsection{Zone IDs and time windows}
\label{sec:zone_id}

In Section \ref{sec:IO_for_TSP}, we describe how IO can be used to learn drivers' preferences from R-TSP examples. Although the Amazon Challenge training dataset consists of 6,112 historical routes, it is difficult to learn any meaningful preference of the drivers at the \textit{stop level} (i.e., individual customer level), since the latitude and longitude coordinates of each stop have been anonymized and perturbed to protect the privacy of delivery recipients \cite{merchan20222021}. However, recall that each stop in the dataset is assigned a Zone ID, which refers to a geographical zone in the city (see Figure \ref{fig:dataset_description}), and each zone contains multiple stops. Analyzing how the human drivers' routes relate to these zones, a critical observation can be made: in the vast majority of the examples, the drivers visit all stops within a zone before moving to another zone (the zone ID of consecutive stops is the same around 85\% of the time). This behavior is illustrated in Figure \ref{fig:driver_stops}. Also, the same zone is usually visited in multiple route examples in the dataset. Thus, instead of learning drivers' preferences at the stop level, we can learn their preferences at the \textit{zone level}. In other words, we consider each zone as a \textit{hypernode} containing all stops with the same Zone ID. Thus, we can create a \textit{hypergraph} with nodes corresponding to the zone hypernodes (see Figure \ref{fig:driver_route_example}). This way, we can view the expert human routes as routes over zones, and we can use our IO approach to learn the weights the drivers use for the edges between zones.

\begin{figure}
\centering
\captionsetup[subfigure]{width=0.96\linewidth}%
    \begin{subfigure}[t]{0.45\linewidth}
        \includegraphics[width = \linewidth]{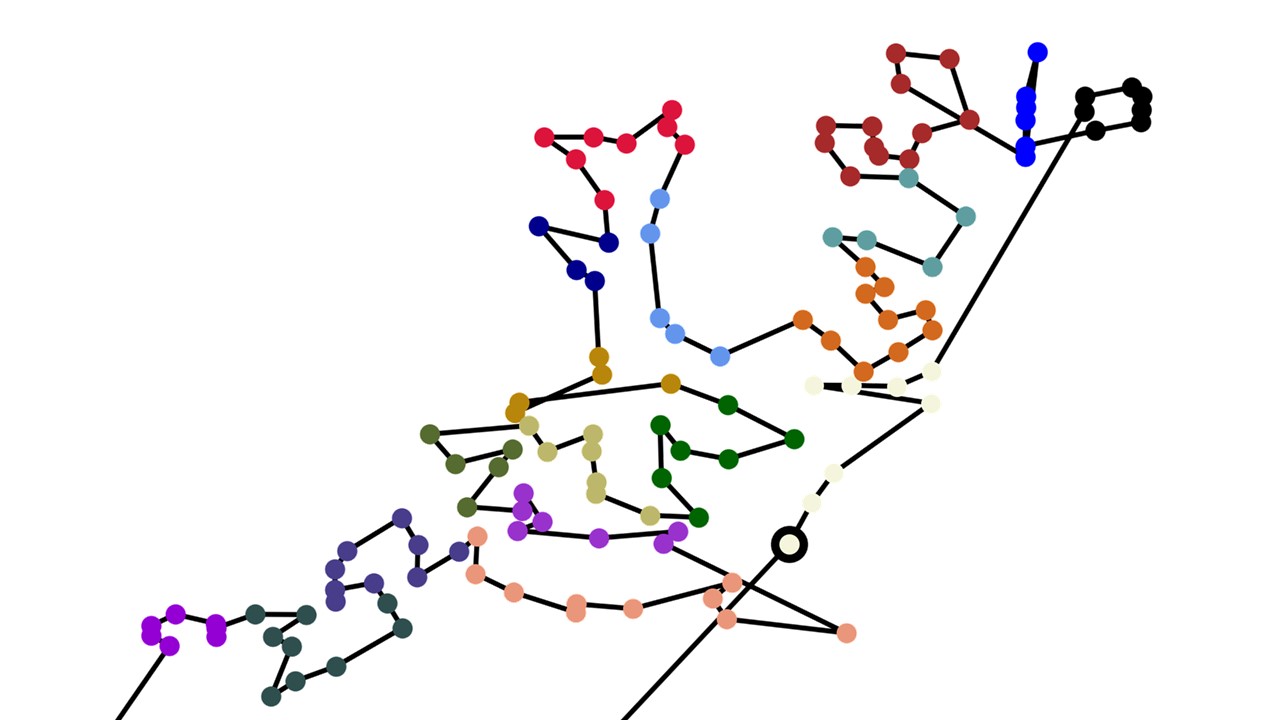}
        \caption{Different colors represent different zones.}
        \label{fig:driver_stops}
    \end{subfigure}
    \begin{subfigure}[t]{0.45\linewidth}
        \includegraphics[width = \linewidth]{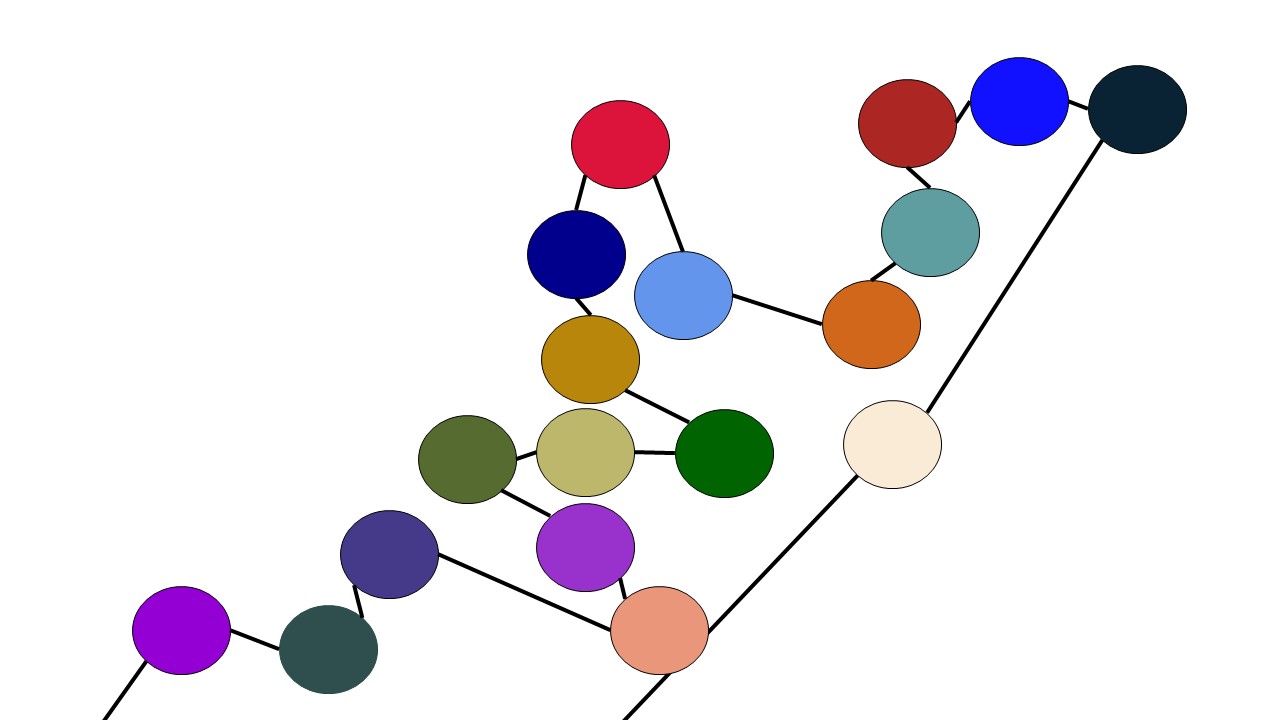}
        \caption{Each zone is substituted by a \textit{hypernode} containing all stops within it.}
        \label{fig:driver_zones}
    \end{subfigure}
\caption{Example of an expert human route from the dataset in terms of its stop sequence and zone sequence.}
\label{fig:driver_route_example}
\end{figure}

Another piece of information in the dataset is that time window targets for package delivery are included for a subset of the stops. These constraints are often trivially satisfied, and ignoring them altogether had minimal impact on the final score of our approach. This was also observed by other contestants of the Amazon Challenge \cite{arslan2021data, cook2022constrained}. Therefore, time windows are ignored in our approach. Moreover, we also ignore all information about the size of the vehicle and the size of the packages to be delivered, as these do not seem to influence the routes chosen by the drivers.

\subsection{Complete method}
\label{sec:complete_method}

In this section, we outline all the steps involved in our IO approach to the Amazon Challenge. As explained in the previous section, due to the nature of the provided data, we focus on learning the preferences of the driver at the zone level. However, the historical routes of the datasets are given in terms of a sequence of stops. Moreover, given a new request for stops to be visited, the learned model should return the sequence of stops, not the sequence of zones. Therefore, intermediate steps need to be taken to go from a sequence of stops to a sequence of zones, and vice-versa. A block diagram of our method is shown in Figure \ref{fig:complete_method}. A detailed description of each step of our method is given in the following.

\textbf{Step 1 (pre-process the data).} The first step is to transform the datasets from stop-level information to zone-level information. Namely, for each data pair of stops to be visited $\hat{s}$ and respective expert route $\hat{x}$ (see R-TSP modeling in Section \ref{sec:IO_for_TSP}), we transform them into a signal $\hat{s}^\text{z}_t$ containing the zones to be visited and respective expert zone sequence $\hat{x}^\text{z}_t$. This is the process illustrated in Figure \ref{fig:driver_route_example}. However, differently from Figure \ref{fig:driver_stops}, there are cases in the dataset where the human driver visits a certain zone, leaves it, and later returns to the same zone. Thus, to enforce that the sequence of zones respects the TSP constraint that each zone is visited only once, when transforming a sequence of stops into a sequence of zones, we consider that a zone is visited at the time the \textit{most consecutive stops} in that zone are visited. To illustrate it, consider the case when a driver visits $7$ stops belonging to zones $A,B,C$, where the sequence of visited stops, in terms of their zones, is $A \to B \to B \to A \to A \to C \to C$. In this case, the driver visits zone $A$, leaves it, and then visits it again. Following our transformation rule, we consider the sequence of zones to be $B \to A \to C$.

\begin{figure}
    \centering
    \includegraphics[width=0.9\linewidth]{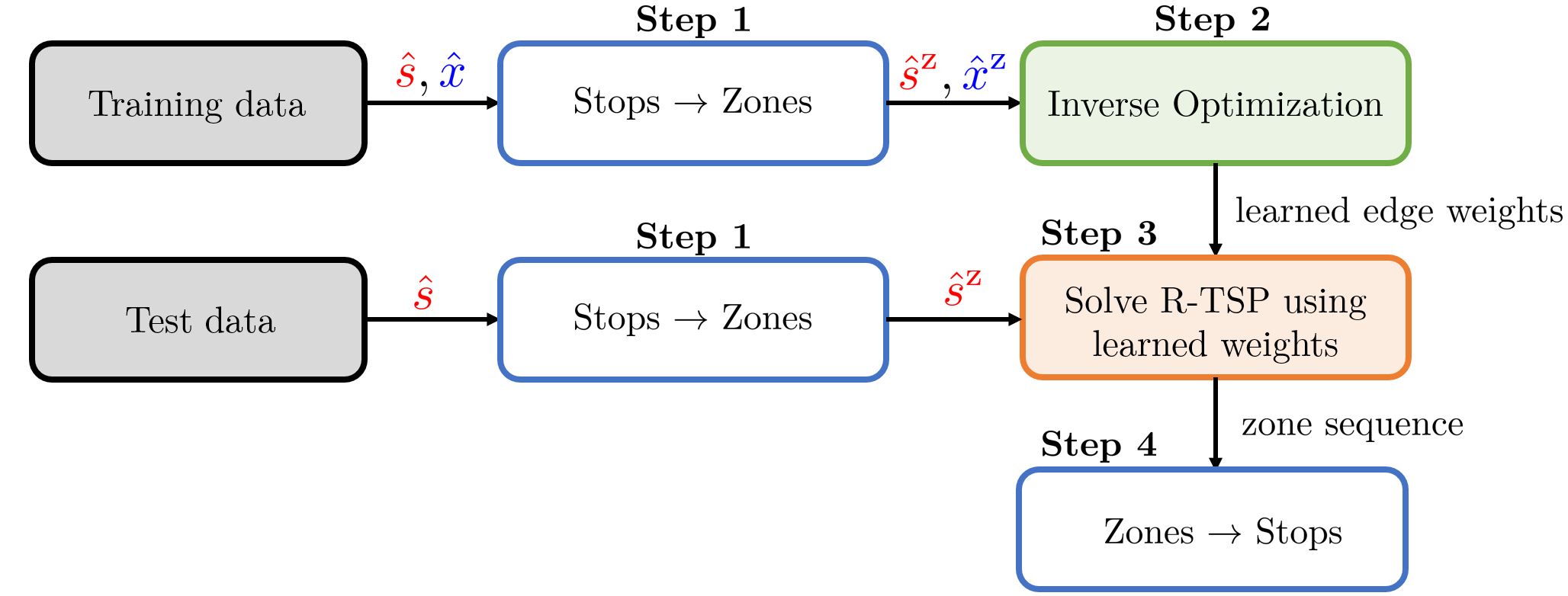}
    \caption{Overview of the proposed method.}
    \label{fig:complete_method}
\end{figure}

\textbf{Step 2 (Inverse Optimization).} Next, considering the hypergraph of zones (i.e., each node represents a zone), we use our IO approach to learn the weights the expert drivers give to the edges connecting the zones. Namely, given a dataset of $N$ examples of zones to be visited and respective zone sequences, we model the problem as an IO problem as in Section \ref{sec:IO_for_TSP}, and we solve Problem \eqref{eq:loss_minimization} using Algorithm \ref{alg:first_order} to learn a cost vector $\theta$, that is, a vector with components corresponding to the learned edge weights between zones.

\textbf{Step 3 (compute the zone sequence).} Let $\hat{s}$ be a set of stops to be visited from the test dataset. To use the weights learned in Step 2 to construct a route for these stops, we first need to transform the signal from the stops to be visited into the \textit{zones} that need to be visited by the driver $\hat{s}^\text{z}$ (Step 1). Given the signal of zones to be visited, and the weights $\theta$ learned in Step 2, we solve the R-TSP over zones with \eqref{eq:TSP_hypothesis} as the cost function and \eqref{eq:TSP_set} as the constraint set. Specific choices for $M_{ij}$ will be discussed in Section \ref{sec:IO_amazon}. The solution to this problem contains the sequence of zones the driver needs to follow. In some cases, routes in the test dataset contain zones that are not visited in the training dataset. In these cases, since the vector of learned weights $\theta$ does not contain information about these zones, we set their weights equal to the Euclidean distance between the center of the zones.

\textbf{Step 4 (from a zone sequence to a stop sequence).} The final step of our method consists of computing the complete route at the stop level. In other words, given the zone sequence computed in Step 3 (e.g., Figure \ref{fig:driver_zones}), we want to find a respective stop sequence (e.g., Figure \ref{fig:driver_stops}). We do it using a penalization method. Let $c_{ij}$ be the transit time from stop $i$ to stop $j$ (this information is provided in the Amazon Challenge dataset, see Figure \ref{fig:dataset_description}). To enforce the zone sequence found in Step 3, we create the penalized weights $\tilde{c}_{ij}$, defined as
$$
\tilde{c}_{ij} \coloneqq
\begin{cases}
  c_{ij}, & \text{if stops $i$ and $j$ are in the same zone} \\
  c_{ij} + R, & \text{if the zone of stop $j$ should be visited directly after the zone of stop $i$} \\
  c_{ij} + 2R, & \text{otherwise},
\end{cases}
$$
where $R > 0$ is a penalization constant. For a large enough $R$, this modification ensures that all stops within a zone are visited before moving to another zone and that the sequence of zones from Step 3 is respected. Thus, we compute the complete route over a set of stops $\hat{S}$ by solving the R-TSP over stops
$$
\min_{x \in \mathbb{X}(\hat{S})} \sum_{i=1}^m \sum_{j=1}^m \tilde{c}_{ij} x_{ij},
$$
where $\mathbb{X}$ is the R-TSP constraint set \eqref{eq:TSP_set} and $m$ is the total number of stops.

As explained in Section \ref{sec:challenge}, the dataset comprises example routes from 5 cities in the USA, and each city can have multiple depots. It turns out that each zone is always served by the same depot, thus, we can learn the preferences of the drivers separately for each depot. Consequently, when using our approach in the Amazon Challenge, we perform steps 1 to 4 separately for each depot. More details on the number of zones served by each depot can be found in Section \ref{sec:complexity} and \cite{merchan20222021}.


\section{Numerical Results}
\label{sec:numerical}

In this section, we numerically evaluate our Inverse Optimization approach to the Amazon Challenge. To compute the zone sequence (i.e., step 3 of our method) we use a Gurobi-based TSP solver \cite{gurobi} (except for the experiments in Section \ref{sec:complexity}) and to compute the complete route at the stop-level (i.e., step 4 of our method), we use the LKH-3 solver \cite{helsgaun2017extension}. The difference between the two is that the Gurobi-based TSP solver is exact, but usually slower for large TSPs, whereas the LKH-3 solver is approximate, but usually faster. Thus, the choice of which solver to use is based on the size of the TSP problem that has to be solved. In our IO approach to the Amazon Challenge, the TSP problem over zones is usually a relatively small one (less than 50 zones), so we solve it using the exact Gurobi-based solver. On the other hand, the TSP problem over stops is usually a larger one (+100 stops), so we solve it using the LKH-3 solver (solving it using the Gurobi-based solver led to little to no improvement in the final Amazon score, while taking significantly more time). Our experiments are reproducible, and the underlying source code is available at \cite{zattoniscroccaro2023amazon}. In particular, we use the InvOpt python package \cite{zattoniscroccaro2023invopt} for the IO part of our approach.

\subsection{IO for the Amazon Challenge}
\label{sec:IO_amazon}

In this section, we present results for two IO approaches: a general approach and the tailored approach proposed in this paper. For both approaches, we use $\eta_t = 0.0005/t$ and $\theta_1$ (that is, the initial point of the used algorithm) as the Euclidean distance between the center of the zones in the training dataset, where we compute the center of a zone by taking the mean of the longitudinal and lateral coordinates of all stops within the zone. All scores reported in this section are the Amazon score of the learned model evaluated in the test dataset.

\textbf{General IO approach.} As a benchmark for our tailored IO methodology, we apply a general IO methodology to the Amazon Challenge. This can be interpreted as using the general IO methodology developed in \cite{zattoniscroccaro2023learning}. In particular, we have the following design choices:
\begin{itemize}
    \item \textit{Hypothesis function:} We use a linear hypothesis, that is, \eqref{eq:TSP_hypothesis} with $M_{ij} = 0$ for all $i,j \in V$.
    \item \textit{IO algorithm:} We use the SAMD algorithm from \cite{zattoniscroccaro2023learning}, with $T=5N$, $\omega (\theta) = \frac{1}{2}\|\theta\|_2^2$, $\Theta = \{\theta : \theta \geq 0\}$.
\end{itemize}
The final Amazon score of the learned IO model is \textbf{0.0535}. This score ranks 11th compared to the 48 models that qualified for the final round of the Amazon Challenge \cite{amazon2021last}. Although this is already a good result, we can significantly improve this score by using our IO approach tailored to routing problems.

\textbf{Tailored IO approach.} To apply our tailored IO approach to the Amazon Challenge, we have the following design choices:
\begin{itemize}
    \item \textit{Hypothesis function:} We use the affine hypothesis \eqref{eq:TSP_hypothesis}. The weights $M_{ij}$ of the affine term are defined below.
    \item \textit{IO algorithm:} We use our tailored Algorithm \ref{alg:first_order}, with standard update steps and $T=5$.
\end{itemize}

As also noticed by some of the contestants of the original Amazon Challenge \cite{winkenbach2021technical}, by carefully analyzing the sequence of zones followed by the human drivers, one can uncover patterns that can be exploited. These patterns are related to the specific encoding of the \textit{Zone ID} given to the zones. Namely, the Amazon Zone IDs have the form \textit{W-x.yZ}, where $W$ and $Z$ are upper-case letters and $x$ and $y$ are integers. Table \ref{table:zone_seq} shows an example of a zone sequence from the Amazon Challenge dataset. Although the zone sequence shown in Table \ref{table:zone_seq} is just a small example, it contains the patterns that we exploit to improve our approach, which are the following:

\begin{itemize}
    \item \textit{Area sequence}: For a zone with Zone ID \textit{W-x.yZ}, define its \textit{area} as \textit{W-x.Z}. It is observed that the drivers tend to visit all zones within an area before moving to another area.
    
    \item \textit{Region sequence}: For a zone with Zone ID \textit{W-x.yZ}, define its \textit{region} as \textit{W-x}. It is observed that the drivers tend to visit all areas within a region before moving to the next region.
    
    \item \textit{One unit difference:} Given two zone IDs $z_1 = \textit{W-x.yZ}$ and $z_2 = \textit{A-b.cD}$, we define the difference between two zone IDs as $d(z_1, z_2) \coloneqq |\text{ord}(W)-\text{ord}(A)| + |x-b| + |y-c| + |\text{ord}(Z)-\text{ord}(D)|$, where the function \texttt{ord} maps characters to integers (in our numerical results, we use Python's built-in \texttt{ord} function). In particular, letters that come after the other in the alphabet are mapped to integers that differ by 1, e.g., $\text{ord(G)} = 71$ and $\text{ord(H)} = 72$. It is observed that for subsequent zone IDs in the zone sequences from the Amazon dataset, the difference between these zone IDs tends to be small (most often 1).
\end{itemize}

\begin{table}
\centering
\begin{tabular}{c | c c c c c c c c c c c c c c c c} 
$W$ & G & G & G & G & G & G & G & G & G & G & G & H & H & H & H & H \\ 
$x$ & 1 & 1 & 1 & 1 & 1 & 1 & 1 & 1 & 1 & 1 & 1 & 1 & 1 & 1 & 1 & 1 \\
$y$ & 1 & 2 & 3 & 2 & 1 & 1 & 2 & 3 & 3 & 2 & 1 & 1 & 2 & 3 & 3 & 2 \\
$Z$ & E & E & G & G & G & H & H & H & J & J & J & A & A & A & B & B \\
\end{tabular}
\caption{Part of the zone sequence from the example route with RouteID f6cf991e-9bb0-46b9-a07d-8192c2d29bb1.}
\label{table:zone_seq}
\end{table}

Next, we incorporate these observations into our IO learning approach. One way to force the routes from our IO model to respect these behaviors (i.e., the ``area sequence'', ``region sequence'' and ``one unit difference'' behaviors) is to use penalization terms. In a sense, using these penalizations can be interpreted as modifying what we believe is the optimization problem the expert human drivers solve to compute their routes. Thus, we use \eqref{eq:TSP_hypothesis} as our hypothesis function, with $M_{ij} = M^A_{ij} + M^R_{ij} + M^d_{ij}$, where $M^A_{ij} = 0$ if zones $i$ and $j$ are in the same area, and $M^A_{ij} = 1$ otherwise, $M^R_{ij} = 0$ if zones $i$ and $j$ are in the same region, and $M^R_{ij} = 1$ otherwise, and $M^d_{ij} = d(i, j)$, that is, the difference between zones $i$ and $j$. Since for Algorithm \ref{alg:first_order} we initialize $\theta_1$ as the Euclidean distance between zone centers, where the coordinates of the centers are given by their latitudes and longitudes, each component of $\theta_1$ is much smaller than 1. This makes a penalization of one unit (such as the ones used for $M^A_{ij}$ and $M^R_{ij}$) enough to enforce that the resulting routes will respect the area sequence and region sequence behaviors. The same idea applies to the ``one unit variance'' penalization.

\begin{figure}
    \centering
    \includegraphics[scale=0.55]{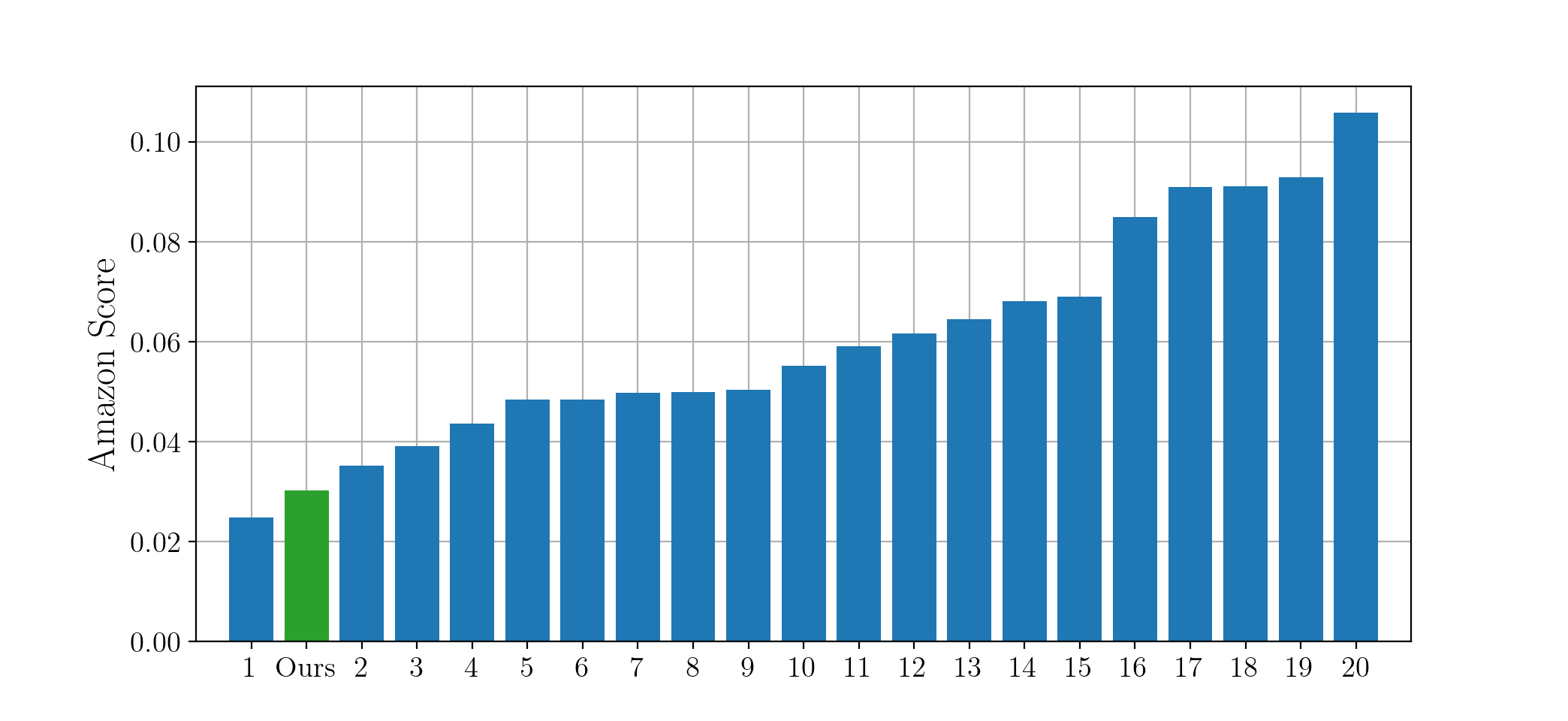}
    \caption{Our final score compared to the scores of the top 20 contestants of the Amazon Challenge.}
    \label{fig:score_leaderboard}
\end{figure}

The final Amazon Challenge score achieved by our tailored approach is \textbf{0.0302}, which significantly improves the 0.0535 score of the benchmark (i.e., general IO) approach. Figure \ref{fig:score_leaderboard} shows the scores of the top 20 submissions of the Amazon Challenge. As can be seen, our score ranks 2nd compared to the 48 models that qualified for the final round of the Amazon Challenge \cite{amazon2021last}. Compared to the initial weights fed to the tailored IO algorithm, considering only the set of weights changed by the first-order method, the change was of $28.6$\% on average, with the 10th and 90th percentiles equal to $0.8$\% and $68.8$\%, respectively. These changes may be interpreted in the following sense: if the first-order algorithm increases the weight of the edge connecting zones \textit{A} and \textit{B}, it means that according to the data, the expert human driver considers this edge more costly than the initial weights (i.e., than the Euclidean distance between the zones), or in other words, the drivers have less preference in using this edge. Similarly, if the algorithm decreases the weight, we can interpret it as the drivers considering this edge less costly, thus, having a stronger preference in using this edge when driving. Moreover, to test the robustness of the learned model, we added Gaussian perturbations to the weights learned and computed the Amazon score of the perturbed model. Adding Gaussian perturbations with magnitudes (in expectation) of $0.1$\% and $1$\% compared to the average magnitude of the weight matrix led to an increase of $0.7$\% and $4$\% in the Amazon score, respectively. Thus, we observed the expected behavior from a robust model: small perturbations lead to small changes in the output of the model.

In our experience, small perturbations on the learned model do not tend to lead to significant changes in the resulting route. To evaluate its robustness, we can add Gaussian perturbations to the weights learned for the Amazon Challenge and compute their effect on the Amazon score of the model. Adding Gaussian perturbations with magnitudes (in expectation) of $0.1$\% and $1$\% compared to the average magnitude of the weight matrix led to an increase of $0.7$\% and $4$\% in the Amazon score, respectively. Thus, we observed the expected behavior from a robust model: small perturbations lead to small changes in the output of the model while increasing the perturbations increases their impact on the model. We have added a discussion on this point to the revised version of the paper (page 21).

\subsection{Computational and time complexity}
\label{sec:complexity}

In this section, we present further numerical experiments using the Amazon Challenge datasets, focusing on the computational and time complexity of Algorithm \ref{alg:first_order}. Before we present our results, as discussed at the end of Section \ref{sec:complete_method}, recall that we apply our IO learning method separately for each depot in the Amazon Challenge training dataset. Thus, assuming we can run Algorithm \ref{alg:first_order} in parallel for all depots, the complexity of computing the final IO model for all depots equals the complexity of computing the IO model for the largest depot in the dataset. For the Amazon Challenge, the largest depot dataset is DLA7 in Los Angeles, which we thus use to discuss the complexity of our approach.

\textbf{Dataset size versus performance.} First, we study the performance of our IO approach by changing the size of the training dataset. That is, instead of using the entire training dataset of the Amazon Challenge to train the IO model, we test the impact of using only a fraction of the available data. Figure \ref{fig:dataset_size_experiment} shows the results of this experiment. Figure \ref{fig:dataset_size} shows the Amazon score achieved, per epoch, by Algorithm \ref{alg:first_order} using different fractions of the Amazon training dataset. Figure \ref{fig:dataset_size_times} shows the time it took to run Algorithm \ref{alg:first_order} for 5 epochs, for the different fractions of the training dataset. As expected, the more data we feed to Algorithm \ref{alg:first_order}, the better the score gets, and the longer the training takes. Interestingly, notice that using only $20\%$ of the data provided for the challenge, our IO approach is already able to learn a routing model that scores $0.0334$, which would still rank 2nd compared to the 48 models that qualified for the final round of the Amazon Challenge.

\begin{figure}
\centering
\captionsetup[subfigure]{width=0.96\linewidth}%
    \begin{subfigure}[t]{0.45\linewidth}
        \includegraphics[width = \linewidth]{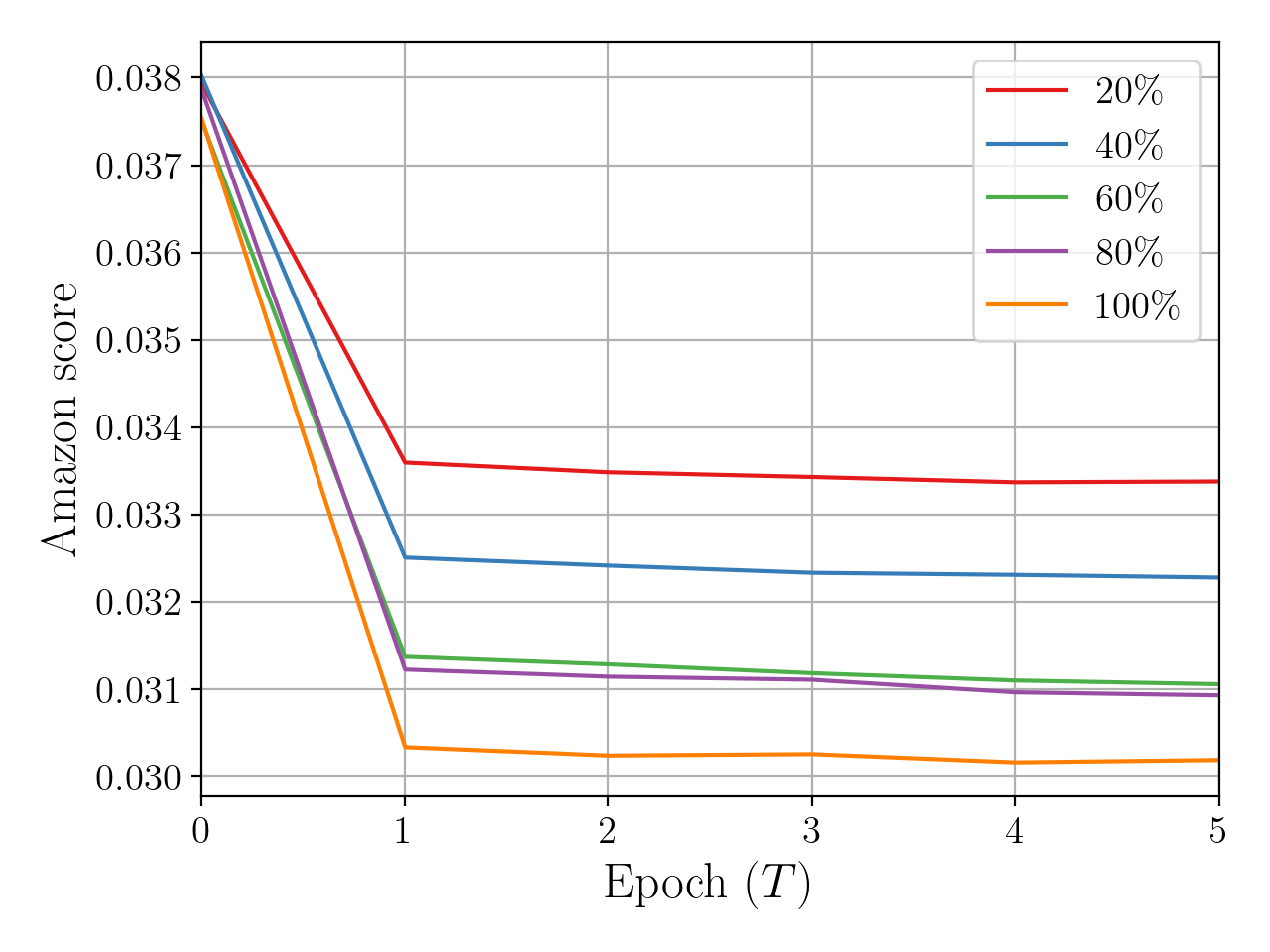}
        \caption{Amazon score on the test dataset, for models learned using different fractions of the training dataset.}
        \label{fig:dataset_size}
    \end{subfigure}
    \begin{subfigure}[t]{0.45\linewidth}
        \includegraphics[width = \linewidth]{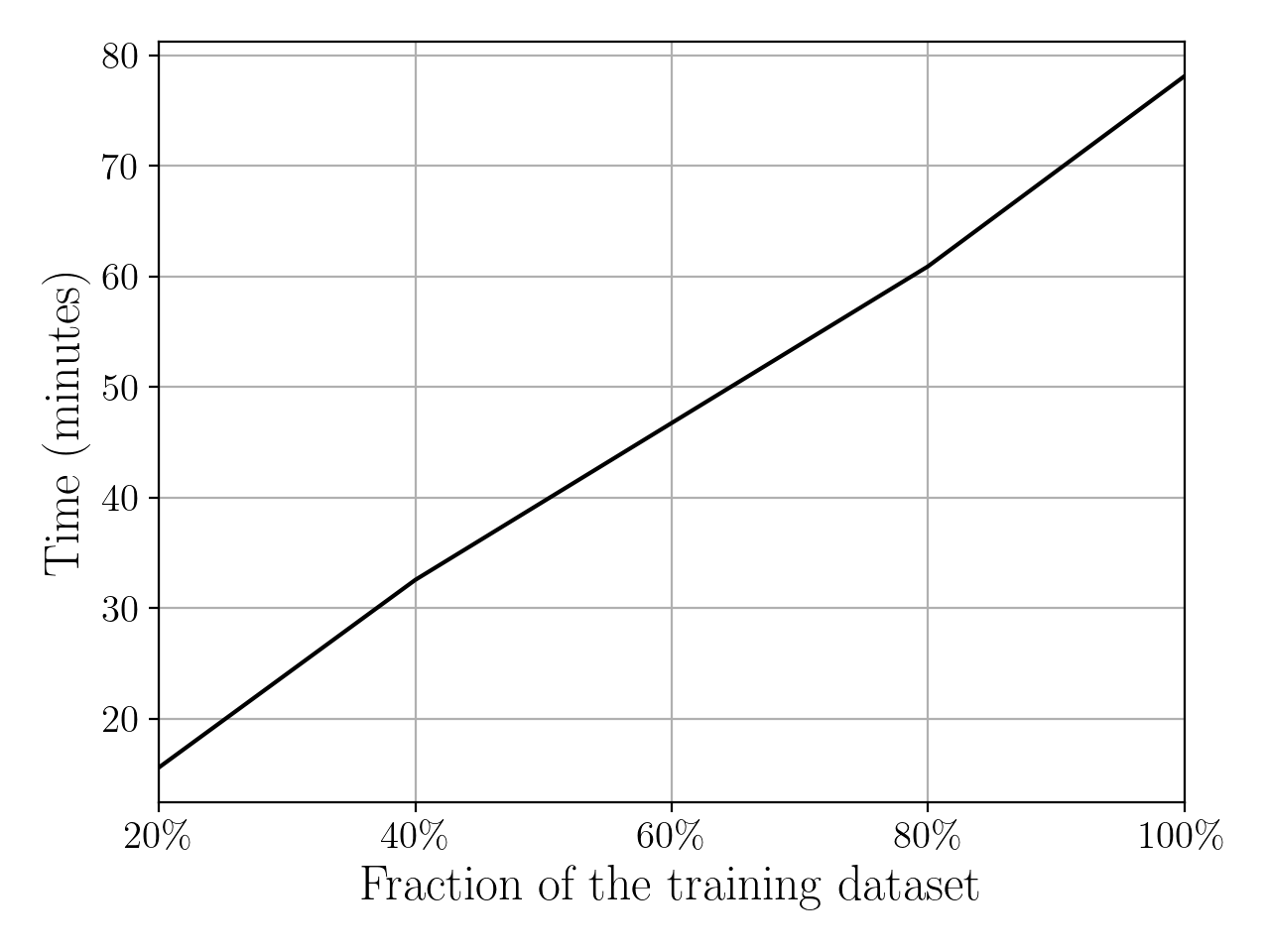}
        \caption{Time taken to run 5 epochs of Algorithm \ref{alg:first_order}.}
        \label{fig:dataset_size_times}
    \end{subfigure}
\caption{Comparison of the Amazon score and learning time of models using different fractions of the Amazon Challenge training dataset.}
\label{fig:dataset_size_experiment}
\end{figure}

\textbf{Time complexity and approximate A-FOP.} In practice, the most time consuming component of Algorithm \ref{alg:first_order} is solving the A-FOP (line 5). As previously explained, for the Amazon Challenge, this problem consists of a TSP over zones (see Step 2 in Section \ref{sec:complete_method}). Thus, for each epoch of Algorithm \ref{alg:first_order}, we need to solve $N$ TSPs, where $N$ is the number of examples in the training dataset. For the depot DLA7, $N = 1133$, and each example contains, on average (rounded up), $23$ zones, where the largest instance has $37$ zones and the smallest has $9$ zones. Thus, for each epoch of Algorithm \ref{alg:first_order}, we need to solve $1133$ TSPs, each with $23$ zones on average. Using an exact Gurobi-based TSP solver, running 5 epochs of Algorithm \ref{alg:first_order} using the entire training dataset took $78.13$ minutes (see Figure \ref{fig:dataset_size_times}).

\begin{table}
\centering
\begin{tabular}{c | c c c c c c c c c c c c c c c c} 
\textbf{TSP solver} & Gurobi & Gurobi & LKH-3 & LKH-3 & OR-Tools & OR-Tools \\
\hline
\textbf{Dataset fraction (\%)} & 20 & 100 & 20 & 100 & 20 & 100 \\
\hline
\textbf{Amazon score} & 0.0334 & 0.0302 & 0.0335 & 0.0302 & 0.0337 & 0.0306 \\
\hline
\textbf{Training time (min)} & 15.59 & 78.13 & 17.87 & 84.62 & 12.72 & 69.51 \\
\end{tabular}
\caption{Summary of the results of Section \ref{sec:complexity}. The Amazon scores are computed using the test dataset. The TSP solver refers to the solver used to solve the A-FOP in line 5 of Algorithm \ref{alg:first_order}, and the training time is the time of running Algorithm \ref{alg:first_order} for 5 epochs.}
\label{table:summary_results}
\end{table}

However, recall that as discussed in Remark \ref{remark:approximate_A-FOP}, Algorithm \ref{alg:first_order} can be used with an \textit{approximate} A-FOP instead of an exact one. The idea here is that solving A-FOP approximately can be faster in practice, which may compensate for a potentially worse performance of the final learned IO model. We test this idea using Algorithm \ref{alg:first_order} with approximate TSP solvers instead of the exact Gurobi-based one. For the approximate solvers, we test the LKH-3 \cite{helsgaun2017extension} and Google OR-Tools \cite{googleOR}. The final Amazon score after 5 epochs of Algorithm \ref{alg:first_order} using Google OR-Tools is $0.0306$, just slightly worse compared to the Gurobi and LKH-3 solvers, but taking only $69.51$ minutes in total. Interestingly, we can push this time even further. As can be seen in Figure \ref{fig:dataset_size}, a good IO model can be achieved using Algorithm \ref{alg:first_order} for only one epoch. Moreover, from Figure \ref{fig:dataset_size}, it can also be seen that a good IO model can be learned using only $20\%$ of the training dataset. Thus, using $20\%$ of the training dataset and running Algorithm \ref{alg:first_order} using the Google OR-Tools TSP solver for 5 epochs, we achieve a final score of $0.0337$ ($0.0341$ after only one epoch) in only $12.72$ minutes (i.e., $2.54$ minutes per epoch on average). This showcases the learning efficiency of our IO methodology, making it also suitable for \textit{real-time} applications, where models need to be learned/updated frequently, and the training time should not take more than a couple of minutes. Table \ref{table:summary_results} summarizes the numerical results of this section.

\subsection{Further Numerical Results}
\label{sec:further_numerical}

\subsubsection{Impact of the initial point}
\label{sec:intial_theta}

An important parameter of Algorithm \ref{alg:first_order} is the initial point $\theta_1^{[1]}$. In practice, the better the initial point, the faster the algorithm will converge, and perhaps more importantly, the better the test dataset performance of the final model tends to be. In this section, we investigate the impact of different choices of $\theta_1^{[1]}$ for the numerical experiment of Section \ref{sec:IO_for_VRPTW} and for the Amazon Challenge. In particular, we compare the ``Euclidean distance'' initialization used to generate the results shown in Figure \ref{fig:vrptw_x_diff_out} and Figure \ref{fig:dataset_size} with a ``uniform'' initialization, where $\theta_1^{[1]}$ is a vector with all its components equal to the same number (this initialization could be used when no prior information on a good cost vector is known). Figure \ref{fig:uniform_theta} shows the results of this experiment. As can be seen, using the Euclidean distance can accelerate the convergence of the algorithm, as in the VRPTW scenario, as well as improve the test dataset performance of the learned model, as in the case of the final Amazon score of the learned models for the Amazon Challenge. This means that, although the Euclidean weights do not explain the routes in the dataset, there is a correlation between the Euclidean distance between nodes and the true weights used to generate the observed routes.

\begin{figure}
\centering
\captionsetup[subfigure]{width=0.96\linewidth}%
    \begin{subfigure}[t]{0.45\linewidth}
        \includegraphics[width = \linewidth]{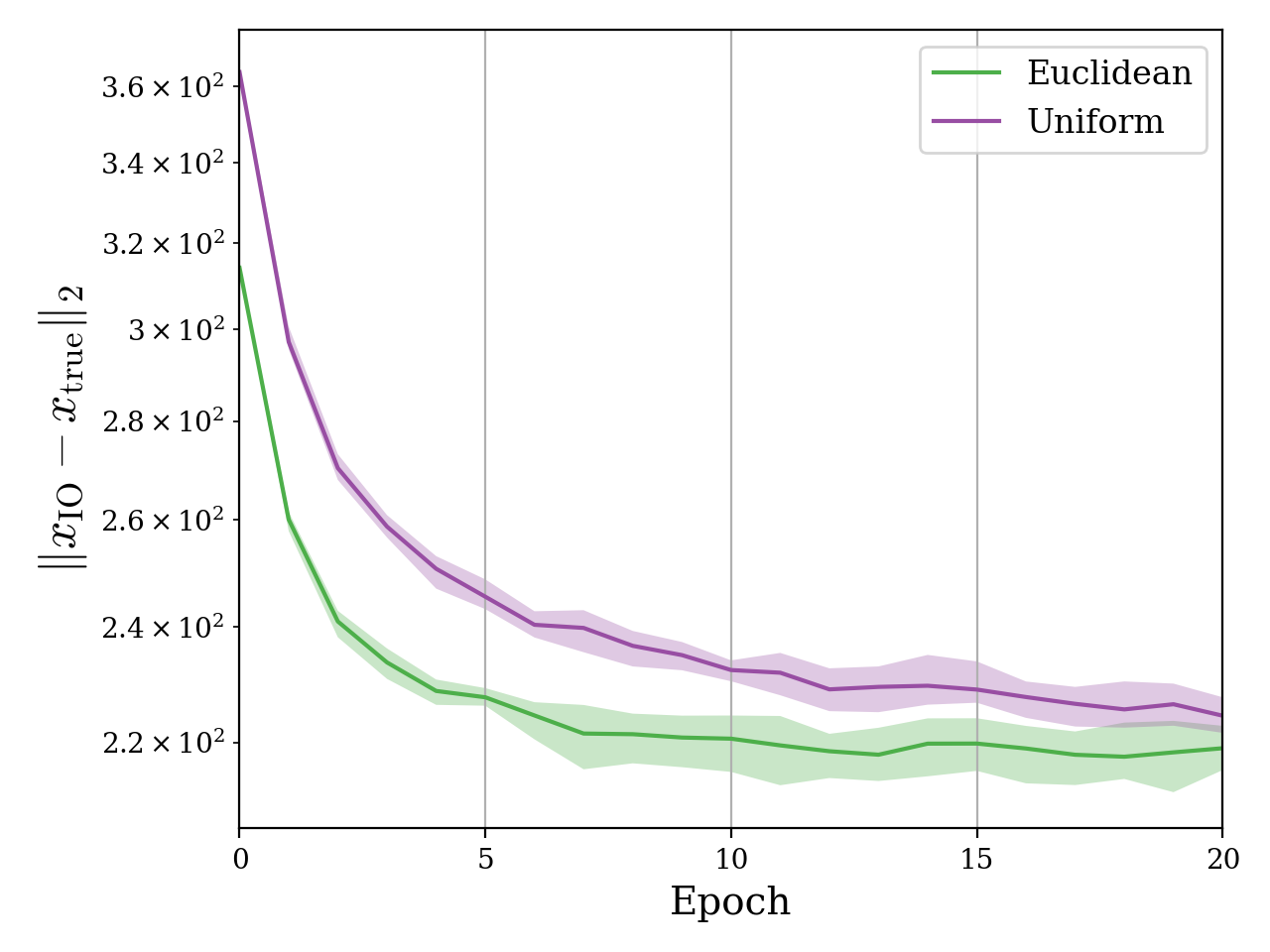}
        \caption{Average error between the routes generated by $\theta_{\text{true}}$ and $\theta_{\text{IO}}$.}
        \label{fig:uniform_theta_vrptw}
    \end{subfigure}
    \begin{subfigure}[t]{0.45\linewidth}
        \includegraphics[width = \linewidth]{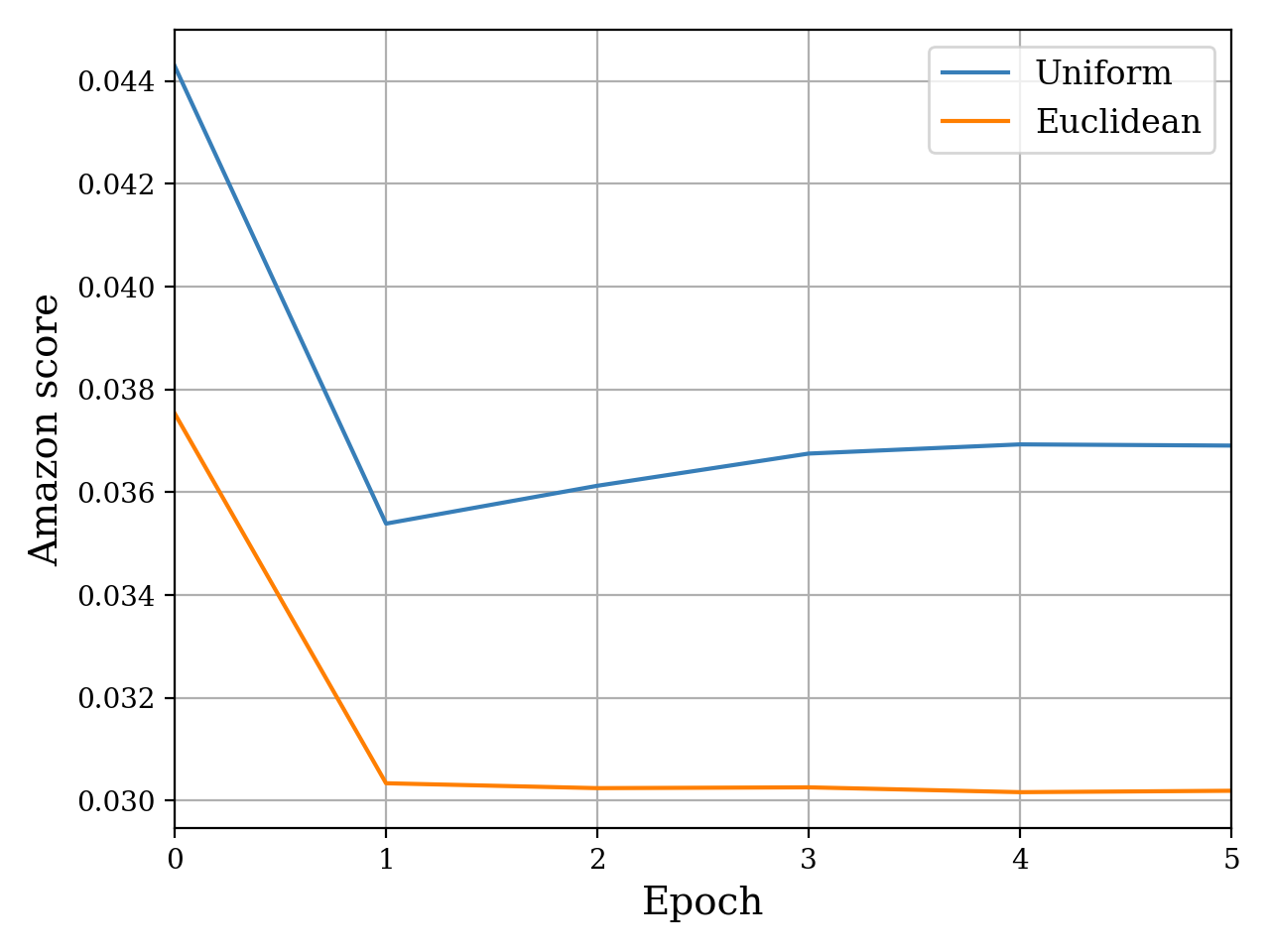}
        \caption{Amazon score of the learned models.}
        \label{fig:uniform_theta_amazon}
    \end{subfigure}
\caption{Comparison between different initializations of $\theta_1^{[1]}$ for Algorithm \ref{alg:first_order}, for the VRPTW scenario of Section \ref{sec:IO_for_VRPTW} and for the Amazon Challenge.}
\label{fig:uniform_theta}
\end{figure}

\subsubsection{Alternative performance metric}
\label{app:error}

\begin{figure}
\centering
\captionsetup[subfigure]{width=0.96\linewidth}%
    \begin{subfigure}[t]{0.45\linewidth}
        \includegraphics[width = \linewidth]{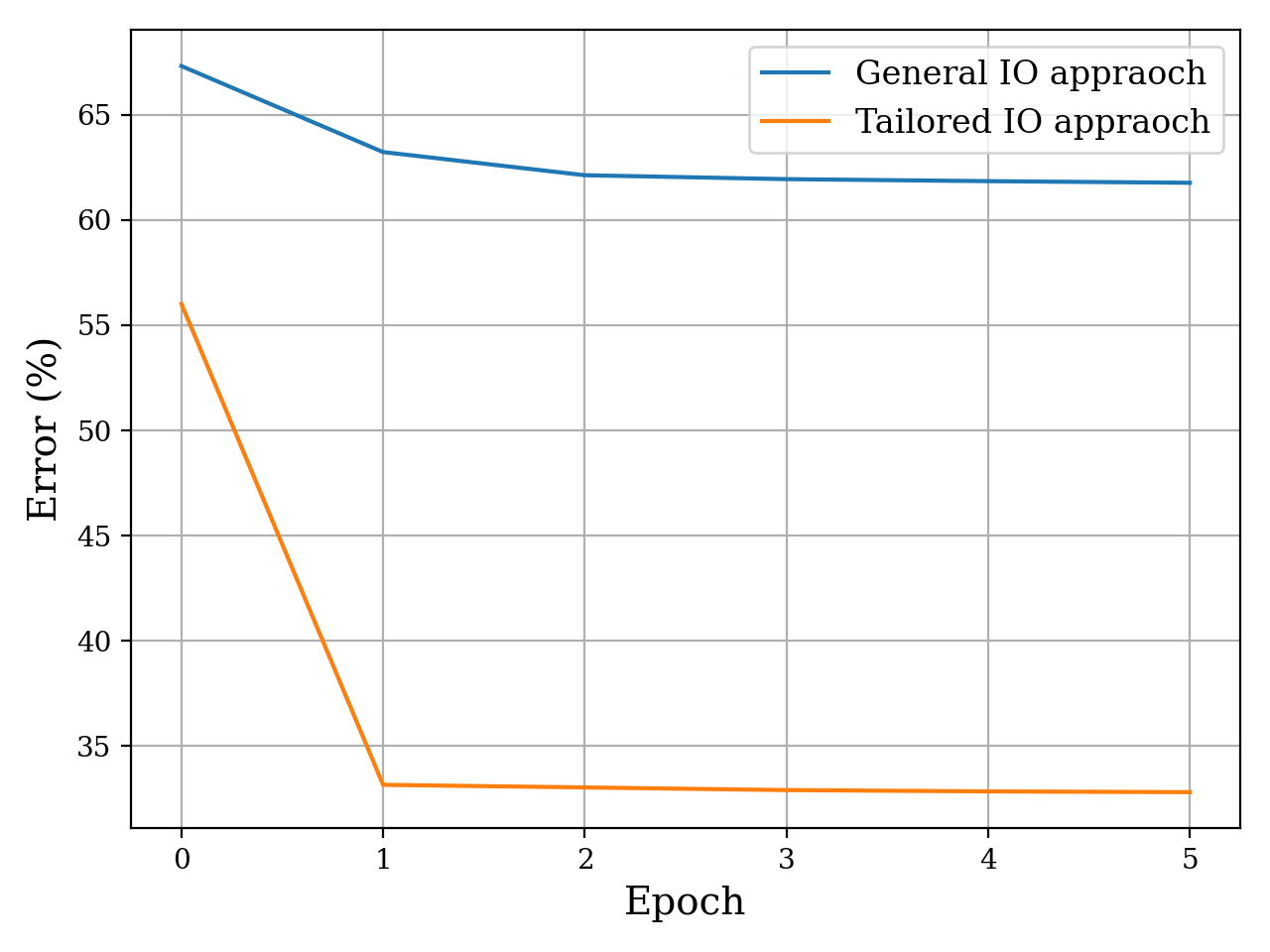}
        \caption{Performance of the general and tailored IO approaches using the zone sequence prediction error.}
        \label{fig:error_amazon}
    \end{subfigure}
    \begin{subfigure}[t]{0.45\linewidth}
        \includegraphics[width = \linewidth]{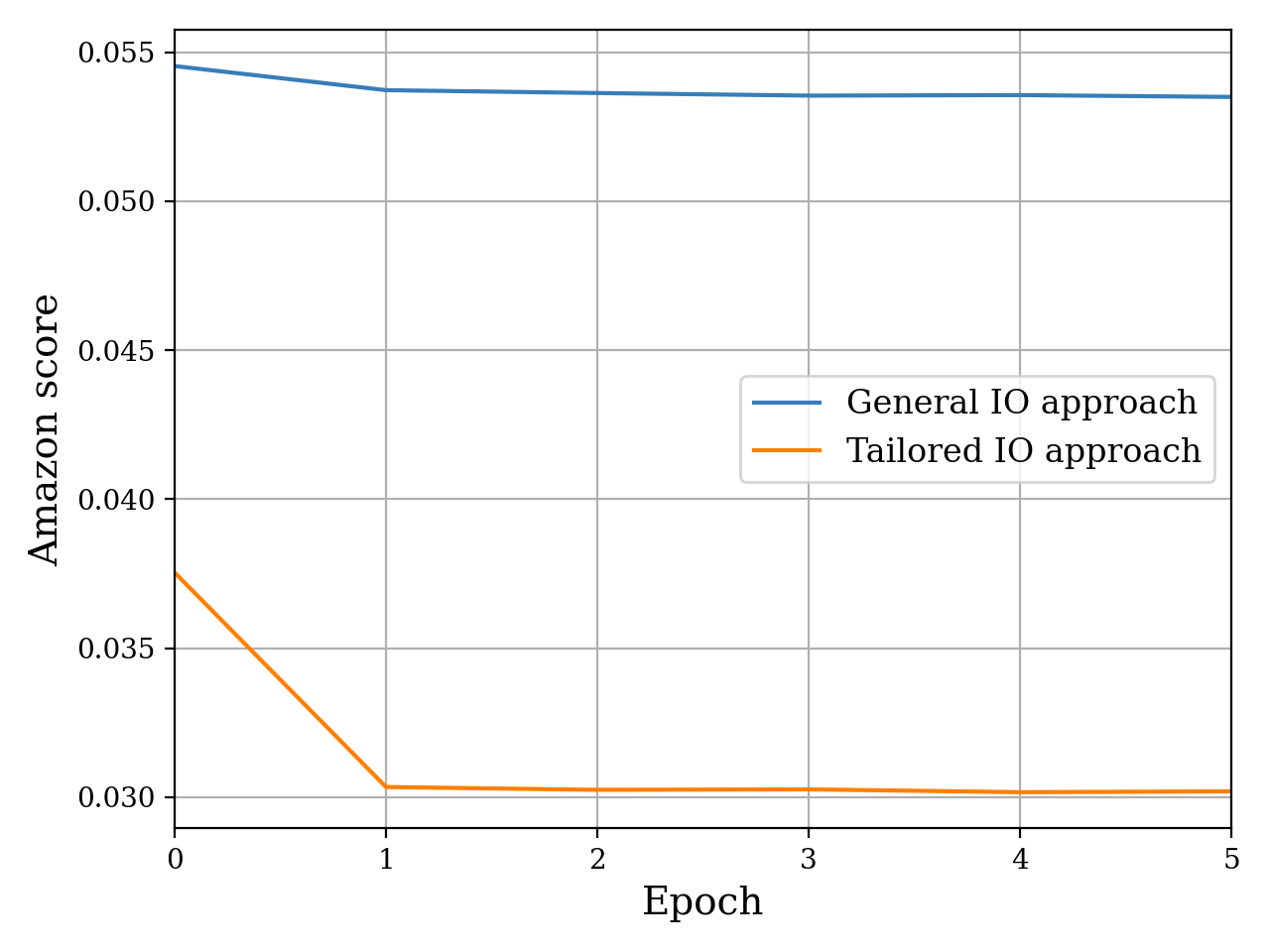}
        \caption{Performance of the general and tailored IO approaches using the Amazon score.}
        \label{fig:score_amazon}
    \end{subfigure}
\caption{Comparison between the zone sequence prediction error and Amazon score performance metrics.}
\label{fig:error_vs_score}
\end{figure}

In Section \ref{sec:numerical}, we evaluated our results for the Amazon Challenge in terms of the Amazon score (see Section \ref{sec:challenge}). In this section, we present results in terms of a zone sequence prediction error metric. Namely, given a zone sequence obtained from a learned IO model (i.e., the output of Step 3 of our IO approach, see Figure \ref{fig:complete_method}), and the zone sequence $\hat{x}$ from the training or test dataset, the prediction error $\text{Error}(x, \hat{x})$ counts how many zones in $\hat{x}$ are in the wrong position compared to $x$. For instance, if $x = \{\text{T-7.1C}, \text{T-7.1B}, \text{T-8.1B}, \text{T-8.1C}, \text{T-8.2C} \}$ and $\hat{x} = \{\text{T-7.1B}, \text{T-7.1C}, \text{T-8.1B}, \text{T-8.2C}, \text{T-8.1C} \}$, then $\text{Error}(x, \hat{x}) = 4$. This performance measure can be interpreted as a generalization of the classical 0-1 error used for classification problems, where $\text{0-1}(x, \hat{x}) = 0$ if $x = \hat{x}$, and $\text{0-1}(x, \hat{x}) = 1$ otherwise. Thus, given a dataset of $N$ examples of zone sequences and the respective sequences predicted by the IO model, we define the total (percentage) zone sequence prediction error across the entire dataset as $100 \sum_{i=1}^N \text{Error}(x^{[i]}, \hat{x}^{[i]}) / \sum_{i=1}^N L^{[i]}$, where $L^{[i]}$ is the length of the $i$'th zone sequence. In other words, this value can be interpreted as the percentage of time the IO approach correctly predicts the position of a zone in the zone sequence. Figure \ref{fig:error_amazon} shows the performance of the general and our tailored IO approaches from Section \ref{sec:IO_amazon} in terms of the zone sequence prediction error. For comparison, we also show their respective Amazon score in Figure \ref{fig:score_amazon}. As can be seen, the IO models show (qualitatively) similar performance, in terms of both prediction error and Amazon score metrics.

\subsubsection{Route examples}
\label{sec:examples}

\begin{figure}
\centering
\captionsetup[subfigure]{width=0.96\linewidth}%
    \begin{subfigure}[t]{0.45\linewidth}
        \includegraphics[width = \linewidth]{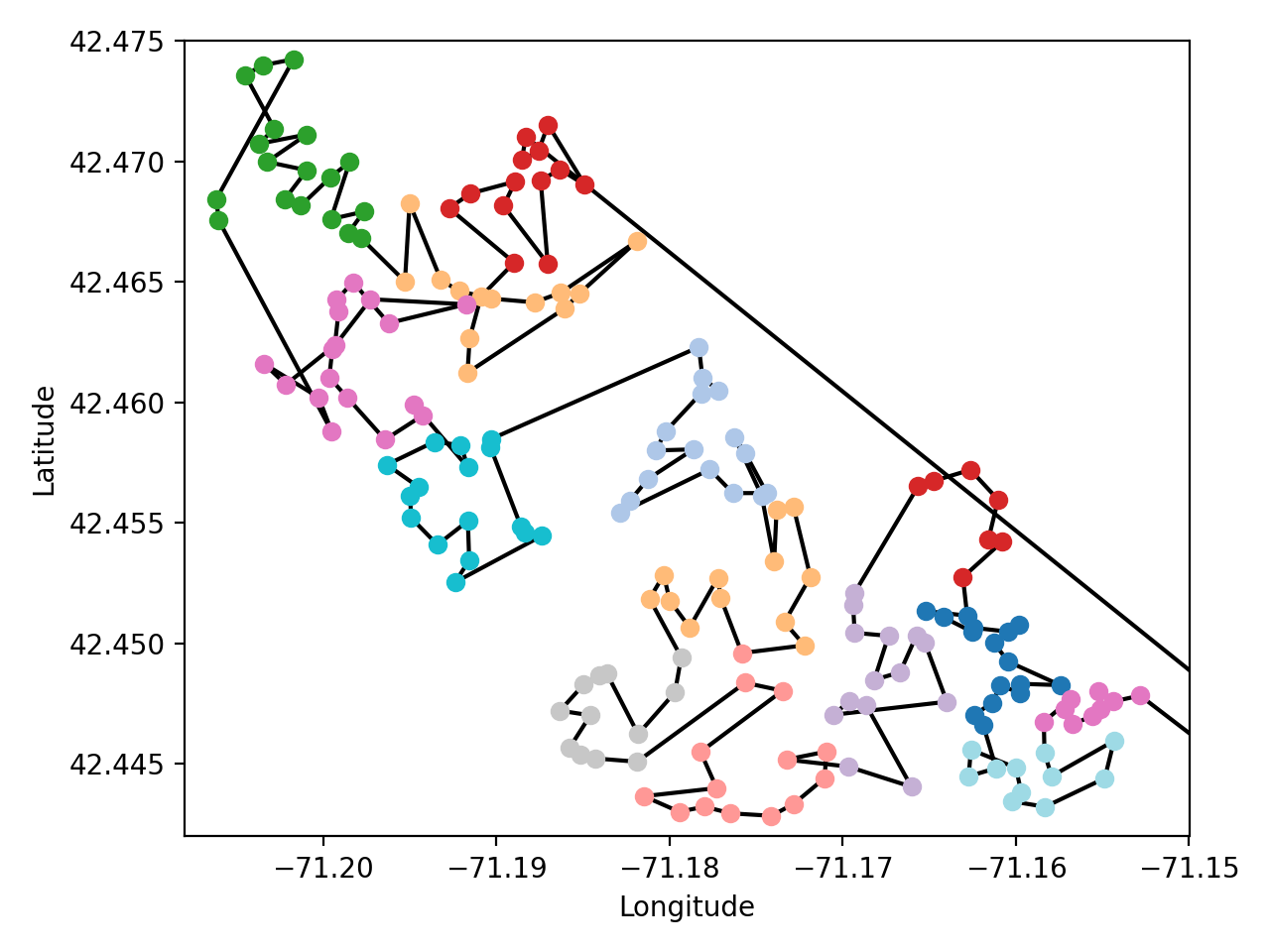}
        \caption{Route example from the test dataset.}
        \label{fig:route_example_1}
    \end{subfigure}
    \begin{subfigure}[t]{0.45\linewidth}
        \includegraphics[width = \linewidth]{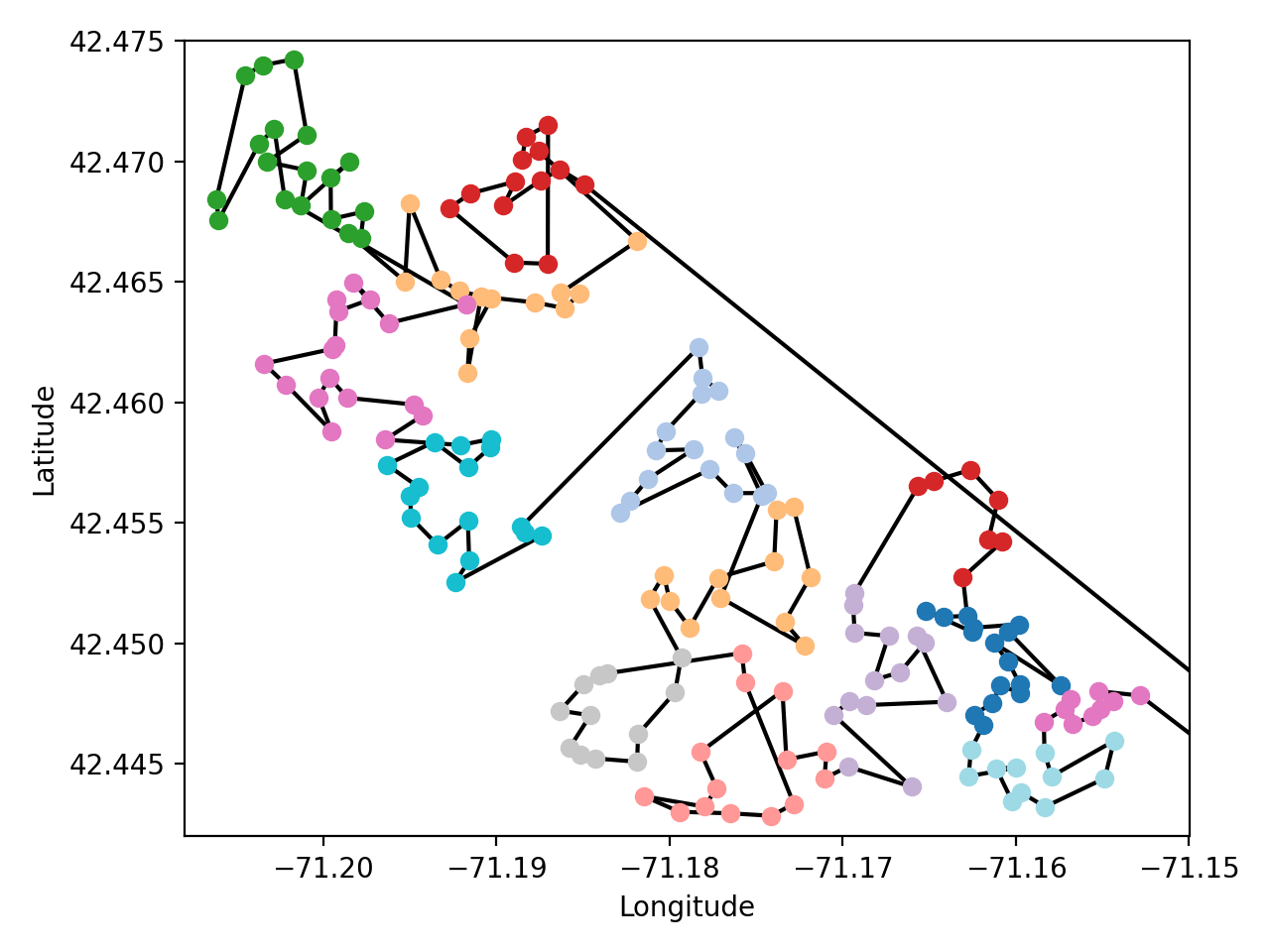}
        \caption{Output route from the IO model.}
        \label{fig:route_example_2}
    \end{subfigure}
\caption{Comparisson between the driver's route from the Amazon Challenge and the output of the IO model. In this example, the zone sequence predicted by the IO model perfectly matches the one from the original route.}
\label{fig:route_example_perfect}
\end{figure}

\begin{figure}
\centering
\captionsetup[subfigure]{width=0.96\linewidth}%
    \begin{subfigure}[t]{0.45\linewidth}
        \includegraphics[width = \linewidth]{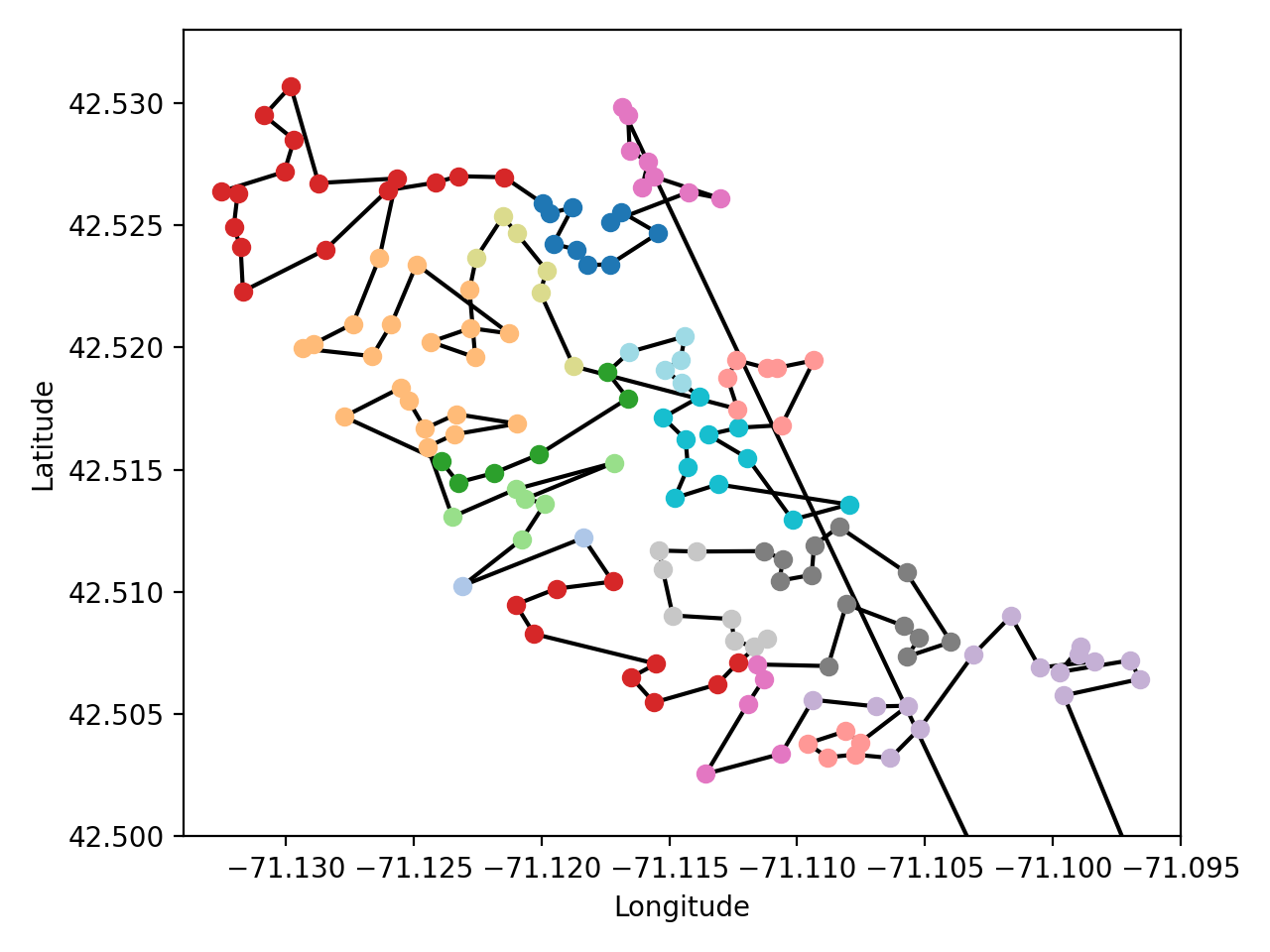}
        \caption{Route example from the test dataset.}
        \label{fig:route_example_3}
    \end{subfigure}
    \begin{subfigure}[t]{0.45\linewidth}
        \includegraphics[width = \linewidth]{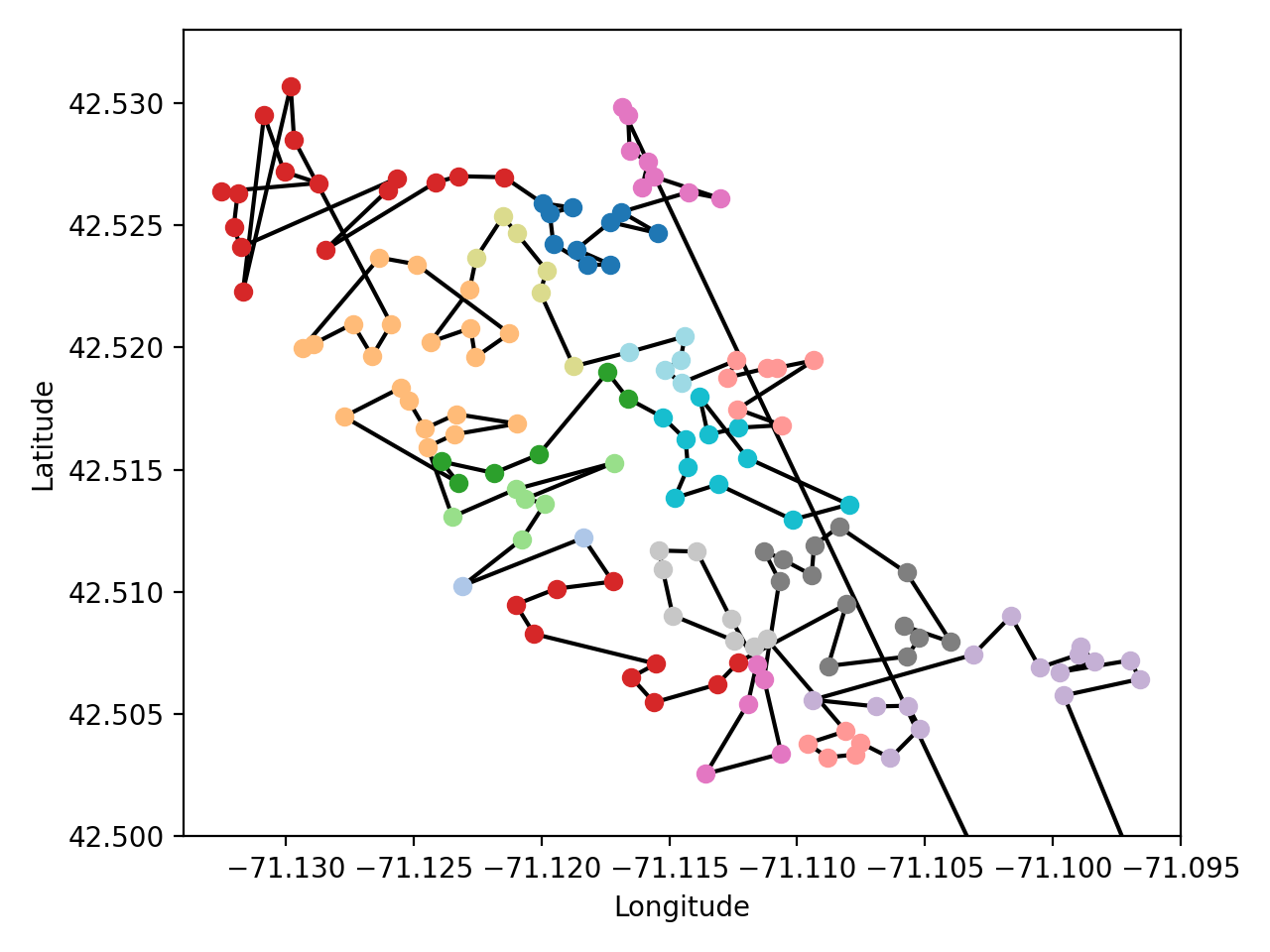}
        \caption{Output route from the IO model.}
        \label{fig:route_example_4}
    \end{subfigure}
\caption{Comparisson between the driver's route from the Amazon Challenge and the output of the IO model. In this example, the zone sequence predicted by the IO model does not match the one from the original route.}
\label{fig:route_example_not_perfect}
\end{figure}

In this section, we show some route examples, comparing the routes of human drivers from the Amazon Challenge dataset, with the routes from our IO approach. In Figure \ref{fig:route_example_perfect}, we show an example where the zone sequence predicted by the IO model (i.e., Step 3 in Section \ref{sec:complete_method}) perfectly matches the one from the original route, where nodes of different colors represent different zones. As can be noticed, even though the zone sequence is the same, the sequence of stops within each zone is different. However, even with these differences, the Amazon score of the route in Figure \ref{fig:route_example_2} is still quite small ($0.0046$). This phenomenon is generally observed for the Amazon Challenge: perfectly predicting the zone sequence tends to lead to a small Amazon score, even with different sequences of stops within each zone. This observation supports our IO approach to the challenge, where we focused on predicting the correct zone sequence, instead of the stop sequences.

In Figure \ref{fig:route_example_not_perfect} we show an example where the zone sequence from the original route (Figure \ref{fig:route_example_3}) differs from the one predicted by the IO model (Figure \ref{fig:route_example_4}). In particular, the zone prediction error (defined in Section \ref{app:error}) between these two routes is $31.6$\%. Still, since the zones predicted in the wrong order are close to each other, the Amazon score of the route in Figure \ref{fig:route_example_4} is relatively small ($0.0117$). This example provides some intuition on the results from Figure \ref{fig:error_vs_score}: even though the average zone prediction error of the proposed tailored IO approach is around $32$\%, the fact that it still guarantees a low Amazon score means that even when predicting the wrong zone sequence, the predicted zones in general similar (i.e., geographically close) to the actual ones from the test dataset.


\section{Conclusion and Further Work}
\label{sec:conclusion}

In this work, we propose an Inverse Optimization (IO) methodology for learning the preferences of decision-makers in routing problems. To exemplify the potential and flexibility of our approach, we first apply it to a simple CVRP problem, where we give insight into how our IO algorithm works by modifying the learned edge weights by comparing the example routes to the optimal route we get using the current learned weights. Then, we apply it to a larger VRPTW example, comparing the performance of our proposed algorithm with different approaches from the literature. Finally, we show the real-world potential of our approach by using it to tackle the Amazon Challenge, where the goal of the challenge was to develop routing models that replicate the behavior of real-world expert human drivers. To do so, we first define what we call Restricted TSPs (i.e., TSPs for which only a subset of the nodes is required to be visited). Given a dataset of signals (nodes to be visited) and expert responses (R-TSPs tours), we have shown how to use IO to learn the edge weights that explain the observed data. In the context of the Amazon Challenge, learning these edge weights translates to learning the sequence of city zones preferred by expert human drivers. Then, from a sequence of zones, we constructed a complete TSP tour over the required stops. The final score of our approach is \textbf{0.0302}, which ranks 2nd compared to the 48 models that qualified for the final round of the Amazon Challenge.

As future research directions, it would be interesting to apply our methodology to different and more complex classes of routing problems, for instance, dynamic VRPs, routing problems with backhauls, as well as routing problems with continuous decision variables. Moreover, although in this work we focused on routing problems, our methodology could also be adapted and tailored to different classes of problems with a binary decision space, such as 0-1 knapsack problems. Given the modularity/flexibility of our IO methodology, we believe it has the potential to be used for a wide range of real-world decision-making problems.


\bibliographystyle{plain}
\bibliography{references}

\end{document}